\documentclass[10pt]{amsart}

\usepackage{amsfonts,amssymb,amsmath,amscd,amstext}
\usepackage[colorlinks=true,linkcolor=blue,citecolor=blue]{hyperref}
\usepackage[utf8]{inputenc}
\usepackage{microtype}
\usepackage{graphicx}
\usepackage{changes}
\usepackage{comment}
\usepackage{mathtools}
\usepackage[a4paper,top=2.5cm,bottom=2.5cm,left=2.5cm,right=2.5cm]{geometry}
\renewcommand{\leq}{\leqslant}
\renewcommand{\geq}{\geqslant}
\renewcommand{\le}{\leqslant}
\renewcommand{\ge}{\geqslant}

\newcommand{\rr}{{\mathbb{R}}}
\newcommand{\G}{{\mathbb{G}}}

\newcommand{\nn}{{\mathbb{N}}}

\newcommand{\Om}{\Omega}
\newcommand{\eps}{\varepsilon}

\renewcommand{\d}{d_{\mathfrak g}}
\newcommand{\ga}{\gamma}

\newcommand{\loc}{\rm loc}
\newcommand{\sub}{\partial_{X,N}u}

\newcommand{\ol}{d_{\sigma^\star}}
\newcommand{\olk}{d_{\sigma_K^\star}}
\newcommand{\otk}{d_{\sigma_{\tilde K}^\star}}
\newcommand{\oth}{d_{\sigma_{\tilde H}^\star}}
\newcommand{\sk}{\sigma_K^\star}

\newcommand{\escpr}[1]{\langle#1\rangle}

\newcommand{\average}{{\mathchoice {\kern1ex\vcenter{\hrule height.4pt
				width 6pt
				depth0pt} \kern-9.7pt} {\kern1ex\vcenter{\hrule height.4pt width 4.3pt
				depth0pt}
			\kern-7pt} {} {} }}
\newcommand{\ave}{\average\int}

\DeclareMathOperator{\lip}{Lip}
\newtheorem{theorem}{Theorem}[section]
\newtheorem{proposition}[theorem]{Proposition}
\newtheorem{lemma}[theorem]{Lemma}

\theoremstyle{definition}

\newtheorem{remark}[theorem]{Remark}

\newtheorem{definition}[theorem]{Definition} 

\theoremstyle{remark}

\numberwithin{equation}{section}

\definechangesauthor[name=Gianmarco, color=brown]{G}
\definechangesauthor[name=Simone, color=purple]{S}
\definechangesauthor[name=Fares, color=blue]{F}

\begin{document}

\title[Monge solutions for discontinuous Hamilton--Jacobi equations in Carnot groups]{Monge solutions for discontinuous Hamilton-Jacobi equations in Carnot groups}

\author[F. Essebei]{Fares Essebei}
\address[Fares Essebei]{ 
Dipartimento di Matematica, 
Universit\`a di Pisa, 
Largo Bruno Pontecorvo 5, 56127 Pisa, Italy }
\email[Fares Essebei]{fares.essebei@dm.unipi.it}
\author[G. Giovannardi]{Gianmarco Giovannardi}
\address[Gianmarco Giovannardi]{Dipartimento di Matematica e Informatica "U. Dini", Università degli Studi di Firenze, Viale Morgani 67/A, 50134 Firenze, Italy}
\email[Gianmarco Giovannardi]{gianmarco.giovannardi@unifi.it}
\author[S. Verzellesi]{Simone Verzellesi}
\address[Simone Verzellesi]{Department of Mathematics, University of Trento, Via Sommarive 14, 38123 Povo (Trento), Italy}
\email[Simone Verzellesi]{simone.verzellesi@unitn.it}

\thanks{G. Giovannardi has been supported by INdAM–GNAMPA 2023 Project \emph{Variational and non-variational problems with lack of compactness}.
S. Verzellesi has been supported by INdAM–GNAMPA 2023 Project \emph{Equazioni differenziali alle derivate parziali di tipo misto o dipendenti
da campi di vettori}.}

\maketitle

\begin{abstract}

  In this paper we study  Monge solutions to  stationary   Hamilton--Jacobi equations associated to discontinuous Hamiltonians in the framework of Carnot groups. After showing the equivalence between Monge and viscosity solutions in the continuous setting, we prove existence and uniqueness for the Dirichlet problem, together with a comparison principle and a stability result.
\end{abstract}
\section{Introduction}
The purpose of this work is to lay out the framework for the study of discontinuous Hamilton-Jacobi equations in the sub-Riemannian setting of Carnot groups. In this regard, we generalise the Euclidean theory by addressing the major challenges implied by the degenerate structure that characterises sub-Riemannian geometry. The study of Hamilton--Jacobi equations plays an important role in modern analysis, and its applications are related to many research areas, e.g. control theory and mathematical physics. The interested reader can find complete surveys of this topic in the monographs \cite{lions,bardicapuzzo,sine}.
The prototypical stationary Hamilton--Jacobi equation is the so-called \emph{eikonal equation}, that is 
\begin{equation}\label{ee}
    |\nabla u|=f(x)
\end{equation}
on $\Om$, where $\Om$ is a domain in $\rr^n$ and $f$ is a continuous function. The study of this kind of equations is typically carried out in the setting of \emph{viscosity solutions} (cf. \cite{CL,users}). Thanks to the effort of many authors (cf. \cite{CL,lions,CraEvaLi,BARDI,bp90,bj90,bcess,ds06} and references therein), problem \eqref{ee} has been generalized by considering first-order differential equations of the general form
\begin{equation*}
    H(x,u,\nabla u)=0
\end{equation*}
on $\Om$, together with their evolutionary counterparts. Here $H$, the so called \emph{Hamiltonian}, is a continuous function which usually satisfies suitable convexity and coercivity properties. A further step has been made by taking into account the case in which the Hamiltonian is not assumed to be continuous (cf. \cite{i85,NewSu,s02,cs03,monge}). In all these papers the authors had to adapt the definition of viscosity solutions taking into account the new measurable setting. In particular, in \cite{NewSu} the authors introduced the notion of \emph{Monge solution} to the eikonal-type equation
\begin{equation*}
    H(\nabla u)=n(x)
\end{equation*}
on $\Om$, where $H$ is convex and continuous and $n$ is lower semicontinuous. The importance of this notion, which is shown by the authors to be equivalent to the viscosity one when $n$ is continuous, is motivated by the fact that the classical \emph{Hopf--Lax formula} (cf. \cite{lions}) does not provide in general a viscosity solution if $n$ is only lower semicontinuous. On the other hand, the setting of Monge solutions is shown to be the right one to establish existence, uniqueness, comparison and stability results. The results in \cite{NewSu} have been later generalized in \cite{monge}, where the authors extended the notion of Monge solution to discontinuous Hamilton--Jacobi equations of the form
\begin{equation}\label{euclhjdavbri}
    H(x,\nabla u)=0
\end{equation}
on $\Om$. Here $H$ is only assumed to be Borel measurable, together with some mild assumptions in the gradient variable. In particular, a crucial hypothesis in \cite{monge} consists in requiring that there exists a positive constant $\beta>0$ such that
\begin{equation}\label{firstcoer}
  H(x,p)\leq 0\quad\implies\quad|p|\leq \beta
\end{equation}
for any $x\in\Om$ and any $p\in\rr^n$. This condition, which can be seen as a weak coercivity requirement, turns out to be fundamental even in the classical viscosity approach. However, as it is well known (cf. e.g. \cite{soravia}), there are many interesting situations in which \eqref{firstcoer} fails. As a significant instance, one can consider the eikonal-type equation
\begin{equation}\label{firsteiksub}
    |\nabla u\cdot C(x)^T|= 1
\end{equation}
on $\Om\subseteq \rr^3$, where $C(x)$ is a $2\times 3$ matrix whose rows encodes the coefficients of the vector fields
\begin{equation*}
    X_1|_q=\frac{\partial}{\partial x_1}+x_2\frac{\partial}{\partial t}\qquad\text{and}\qquad X_2|_q=\frac{\partial}{\partial x_2}-x_1\frac{\partial}{\partial t},
\end{equation*}
where we denoted points $q\in\rr^3$ by $q=(x_1,x_2,t)$. Being the kernel of $C(x)^T$ non-trivial, it is easy to notice that the Hamiltonian associated to \eqref{firsteiksub} does not satisfy \eqref{firstcoer}. A standard approach to overcome this difficulty consists in changing the underlying geometry of the ambient space. Indeed, $X_1$ and $X_2$ make up a particular example of generating vector fields associated to a \emph{Carnot-Carathéodory structure}, that is the sub-Riemannian Heisenberg group $\mathbb H^1$ (cf. \cite{CDPTbook, gromov} for an exhaustive survey of the topic). Therefore \eqref{firsteiksub} can be rephrased by considering the corresponding \emph{sub-Riemannian} equation
\begin{equation}
\label{subeiko}
    |Xu|= 1
\end{equation}
on $\Om$, where $Xu$ denotes the so-called \emph{horizontal gradient} associated to the vector fields $X_1$ and $X_2$ (cf. Section \ref{calculus}). The sub-Riemannian eikonal equation \eqref{subeiko} has been studied in the viscosity setting in \cite{D2007} in general Carnot--Carathéodory spaces, whereas more general equations has been considered for instance in \cite{StrofMan,Strof,Biro,soravia,CGPV,bdm09, BG19,MR3168637}.  In the broad generality of metric spaces the notion of viscosity solution for the Hamilton--Jacobi equation has been studied by \cite{GHN15,AF14}. In \cite{GS15} the authors introduced a different notion of metric viscosity solution for continuous Hamiltonians $H(x,u, |\nabla u|)$ based  on the local metric slope $|\nabla u|$, that is a generalized notion of the gradient norm of $u$ in metric spaces, and they showed several comparison and existence results. Moreover, in \cite{LSZ21} the authors studied the eikonal equation \eqref{ee} in  complete and rectifiable connected metric spaces, providing  the equivalence between their notion of viscosity solutions and Monge solutions when the right hand side $f$ is continuous with respect to the metric distance. To the best of our knowledge, in a general metric space a notion of metric gradient is not available, and only the local metric slope $|\nabla u|$ can be considered (cf. \cite{ags}). Accordingly, the last entry of the metric Hamiltonian is a scalar and not a vector. Indeed, during the preparation of this manuscript, we became aware of the work \cite{LSX}, where the authors study the eikonal equation in metric measure spaces with a discontinuous  inhomogeneous term.  However in Carnot--Carathéodory spaces, that are examples of length metric spaces, the notion of horizontal gradient $Xu$ is well-known and general stationary discontinuous Hamiltonians $H(x,Xu)$ can be considered.

In the present paper, inspired by \cite{monge}, we study sub-Riemannian Hamilton--Jacobi equations of the form
\begin{equation}\tag{H-J}
\label{eq:HJI}
H(x,Xu)=0
\end{equation}
on $\Om$, where here and in the following $\Om$ denotes a subdomain of a \emph{Carnot group} $\mathbb G$ of rank $m$, $Xu$ is the horizontal gradient associated to $\G$ and the Hamiltonian $H:\Om\times\rr^m \longrightarrow \rr$ satisfies the following structural assumptions (H):
\begin{itemize}
\label{H}
\item[$(\text{H}_1)$]  $H: \Om \times \rr^m \to \rr$ is Borel measurable;
\item[$(\text{H}_2)$] The set
\[
Z(x):=\{p\in\rr^m\,:\,H(x,p) \le 0\}
\]
is closed, convex and $\partial Z(x)=\{p\in\rr^m\,:\,H(x,p)=0\}$ for any $x\in\Om$;
\item[$(\text{H}_3)$] There exist $\alpha > 1$ such that 
$$\hat{B}_{\frac{1}{\alpha}}(0) \subset Z(x) \subset \hat{B}_{\alpha}(0)$$ 
for any $x\in\Om$, where $\hat{B}_{\alpha}(0)$ is Euclidean  open  ball of radius $\alpha$ centered at the origin in $\rr^m$.
\end{itemize}
We recall that  a \emph{Carnot group $\G$ of dimension $n$ and rank $m$} is a connected and simply connected $n$-dimensional Lie group whose Lie algebra $\mathfrak g$ of left-invariant vector fields admits a suitable \emph{stratification}, and is generated via Lie brackets by a particular $m$-dimensional vector subspace, the \emph{first layer} $\mathfrak g_1$, for which a basis $X_1,\ldots,X_m$ is fixed. We refer to Section \ref{carnotgroups} for a brief introduction to Carnot groups.
Since the exponential map associated to a Carnot group is a global diffeomorphism (cf. \cite{BonLanUgu}), here and in the following we identify $\G$ with $\rr^n$, with group law inherited by $\G$ by means of the \emph{Campbell--Baker--Hausdorff formula} (cf. \cite{BonLanUgu}). 
In the following, we refer to $X_1,\ldots,X_m$ as \emph{generating horizontal vector fields.} Moreover, we recall that an absolutely continuous curve
$\gamma\colon [0,T]\to \mathbb{G}$ is said to be \emph{horizontal} if there exists $a(t)=(a_1(t),\ldots, a_m(t))\in L^\infty((0,T),\rr^m)$ such that
\begin{equation}\label{curvaorizzontale}
    \dot{\gamma}(t)=\sum_{i=1}^m a_i(t)X_i(\gamma(t))
\end{equation}
 for a.e. $t\in (0,T)$. The structural assumptions (H) allows us to associate a suitable norm to the Hamiltonian $H$. More precisely, inspired by \cite{monge}, we define $\sigma^\star:\Om\times\rr^m\longrightarrow[0,\infty)$ by
\begin{equation}\label{supportfunct}
\sigma^\star(x,p)=\sup \{ \escpr{-\xi,p} \ : \ \xi \in Z(x)\}
\end{equation}
for any $x\in\Om$ and any $p\in\rr^m$. It is easy to observe that $\sigma^\star$ is a sub-Finsler norm defined on the \emph{horizontal bundle} $H\Om$, that is the subbundle of $T\Om$ of  those vector fields, which are called \emph{horizontal}, which are tangent to the $m$-dimensional distribution generated by $X_1,\ldots,X_m$. Accordingly, we exploit $\sigma^\star$ to induce a distance $\ol$ on $\Om$, whose Euclidean counterpart is known in literature as \emph{optical length function}, by
\begin{equation}\label{induceddd}
d_{\sigma^\star}(x,y)= \inf \left  \{ \int_0^1 \sigma^\star(\ga(t),\dot{\ga}(t)) \,dt    \ : \ \ga:[0,1]\longrightarrow\Om,\,\text{$\ga$ is horizontal},\,\ga(0)=x,\,\ga(1)=y \right\}
\end{equation}
for each $x,y \in \Omega$.  Being $\Om$ open and connected, the celebrated Chow-Rashevskii connectivity theorem (cf. \cite{Chow,Nagel,gromov}) implies that $\ol$ is finite for any $x,y\in\Om$,  since every two points in an open and connected set can be joined by a horizontal curve. Again inspired by \cite{NewSu,monge}, we are ready to state the main definition of this paper.
\begin{definition}[Monge solution]\label{mongedefinizione}
Let $\Omega \subset \mathbb{G}$ be an open and connected subset of $\mathbb{G}$. If $u \in C(\Om)$, we say that $u$ is a \emph{Monge solution} (resp. \emph{subsolution},\emph{supersolution}) to 
\eqref{eq:HJ}
in $\Omega$ if 
\begin{equation}
\liminf_{x \to x_0} \frac{u(x) - u(x_0) + d_{\sigma^\star}(x_0, x)}{d_{\Om}(x_0,x)}
= 0 \quad (\text{resp. } \geq, \leq)
\end{equation}
for any $x_0\in\Om$, where $d_\Om$ is the standard \emph{Carnot-Carathéodory distance} on $\Om$ (cf. Section \ref{prel}).
\end{definition}
The aim of this paper is to investigate the main aspects of this definition in the sub-Rimannian setting, recovering the Euclidean results achieved in \cite{monge}. A first step consists in relating this notion to the classical sub-Riemannian notion of viscosity solution. To this aim, after describing some properties of the optical length function \eqref{induceddd} (cf. Section \ref{olprop}) and of viscosity solutions in Carnot groups (cf. Section \ref{viscsect}), we will show that the theory of Monge solutions embeds the theory of viscosity solutions, proving the equivalence of these two notions as soon as the Hamiltonian is continuous.
\begin{theorem}\label{introequiv}
     Let $\Om\subseteq \G$ be a domain. Let $H$ be a continuous Hamiltonian satisfying {\rm (H)}. Then $u\in C(\Om)$ is a Monge subsolution (resp. supersolution) to \eqref{eq:HJI} if and only if it is a viscosity subsolution (resp. supersolution) to \eqref{eq:HJI}.
\end{theorem}
Despite some similarities with the Euclidean method, the sub-Riemannian structure requires  some adjustments.  Indeed, in order to prove Theorem \ref{introequiv}, we will first need to recover a suitable Hopf--Lax formula for the Dirichlet problem associated to \eqref{eq:HJI}. In this respect, the first striking difference with the Euclidean environment emerges. Indeed, in the classical theory of Monge solutions (cf. \cite{NewSu,monge}) the optical length function is defined on the whole $\overline\Om$. This possibility relies on the fact that every two points in $\overline\Om$ can be joined by an Euclidean Lipschitz curve as soon as the boundary of $\Om$ is locally Lipschitz. Unfortunately this property is no longer true in our setting, since it is not always the case that two points on $\partial\Om$ can be connected by a horizontal curve lying in $\overline\Om$. A useful consequence of the Euclidean approach is that the optical length function is a \emph{geodesic distance}  (cf. Section \ref{lentandgeodistancesection}), which is no longer true in our case. The solution to this first major problem relies on some delicate localisation arguments. The key point in which one would like to exploit the fact that the optical length function is defined up to the boundary is the validity of the classical Hopf--Lax formula. To be more precise,  in the Euclidean setting it is the case (cf. \cite[Theorem 5.3]{monge}) that if $\Om$ is a bounded domain with Lipschitz boundary and $g\in C(\partial\Om)$ satisfies the compatibility condition 
\begin{equation*}
    g(x)-g(y)\leq \ol(x,y)
\end{equation*}
for any $x,y\in\partial\Om$, the function 
$w$ defined by 
\begin{equation*}
    w(x)=\inf_{y\in\partial\Om}\{\ol(x,y)+g(y)\}
\end{equation*}
is a Monge solution to  \eqref{euclhjdavbri} and coincides with $g$ on $\partial\Om$. Since our optical length function is defined only on $\Om$, this formula would become meaningless. We overcome this difficulty by suitably extending our original Hamiltonian. To this aim, we let
\begin{equation}\tag{$\mathcal K$}\label{kzero}
    \mathcal K(H,\Om):=\{K:\rr^n\times\rr^m\longrightarrow \rr\,:\,\text{$K$ satisfies {\rm (H)} and $K\equiv H$ on $\Om\times\rr^m$}\}.
\end{equation}
Notice that $\mathcal K(H,\Om)$ is always non-empty, as every Hamiltonian can be extended to the whole $\rr^n\times\rr^m$ by letting $H(x,p)=|\xi|-\alpha$ outside $\Om\times\rr^m$. For any fixed $K\in \mathcal K(H,\Om)$, we consider the associated metric $\sigma^\star_ K$ and optical length function $\olk$.  The advantage of this approach consists in the fact that, in view of the aforementioned Chow--Rashevskii connectivity theorem, every two points in $\rr^n$ can be connected by a horizontal curve. Therefore, $\olk$ is actually a finite distance on the whole $\rr^n$, whence in particular on $\overline\Om$. Surprisingly, we will show (cf. Proposition \ref{mongequivdef}) that the definition of Monge solution on $\Om$ is invariant by replacing $H$ with any $K\in\mathcal K(H,\Om)$. In this way, what \emph{a priori} constitutes a considerable problem ensures \emph{a posteriori} a more accurate understanding of Monge's notion of solution. These facts motivate the following result.
\begin{theorem}[Hopf--Lax formula]\label{mongedir}
    Let $\Om\subseteq \G$ be a domain and let $H$ satisfy {\rm (H)}. Let $g\in C(\partial\Om)$ be  bounded and  such that there exists $K\in\mathcal K(H,\Om)$ for which
    \begin{equation}\tag{BCC}\label{compcond}
        g(x)-g(y)\leq\olk (x,y)
    \end{equation}
    for any $x,y\in\partial\Om$. Let us define 
    \begin{equation}\tag{H-L}\label{hopflax}
    w(x):=\inf_{y\in\partial\Om}\{\olk (x,y)+g(y)\}.    
    \end{equation}
    Then $w\in \lip (\Om,d_\Om)\cap C(\overline\Om)$ and $w$ is a Monge solution to the Dirichlet problem 
\begin{align*}
    H(x,Xw) &= 0 \quad \mbox{in } \Om \\
    w &= g \quad \mbox{on } \partial \Om.
\end{align*}
\end{theorem}
 Notice that the compatibility condition \eqref{compcond} is trivially  necessary for the function $w$ given by \eqref{hopflax} to attain
the boundary datum $g$ on $\partial \Omega$.
After proving Theorem \ref{introequiv} and Theorem \ref{mongedir}, we continue the study of \eqref{eq:HJI} in the discontinuous setting. First, we show the validity of the following comparison principle for Monge solutions. 
\begin{theorem}[Comparison Principle]\label{comparison}
 Let $\Om\subseteq\G$ be a bounded domain.  Let $H$ be a Hamiltonian satisfying {\rm (H)}, let $u \in C(\bar{\Omega})$ be a Monge subsolution to \eqref{eq:HJ} and $v \in C(\bar{\Omega})$ be a Monge supersolution to  \eqref{eq:HJ}. If $u \le v$ on $\partial \Omega$, then $u \le v $ in $\Omega$.
\end{theorem}
Notice that, combining Theorem \ref{mongedir} and Theorem \ref{comparison}, we guarantee existence and uniqueness for the Dirichlet problem associated to \eqref{eq:HJI} under the compatibility condition \eqref{compcond}.
 Finally, inspired by \cite{monge}, we show that the notion of Monge solution is stable under suitable notions of convergence for sequences of Hamiltonians and Monge solutions. 
\begin{theorem}[Stability]\label{stab}
     Let $\Om\subseteq G$ be a domain. Let $(H_n)_{n\in\mathbb N}$ and $H_\infty$ satisfy {\rm (H)} with a uniform choice of $\alpha$. For any $n\in\mathbb N$, let $u_n\in C(\Om)$ be a Monge solution to 
    \begin{equation*}
        H_n(x,Xu_n(x))=0
    \end{equation*}
    on $\Om$. Assume that $d_{\sigma_n^\star}\to d_{\sigma_\infty^\star}$ locally uniformly on $\Om\times\Om$, where, for any $n\in\mathbb N$, $d_{\sigma_n^\star}$ is the optical length function associated to $H_n$ and $d_{\sigma_\infty^\star}$ is the optical length function associated to $H_\infty$. Assume that there exists $u_\infty\in C(\Om)$ such that $u_n\to u_\infty$ locally uniformly on $\Om$. Then $u_\infty$ is a Monge solution to 
    \begin{equation*}
        H_\infty(x,Xu_\infty(x))=0
    \end{equation*}
    on $\Om$.
\end{theorem}
The paper is organized as follows. In Section \ref{prel} we recall some basic facts and properties about Carnot groups and length spaces. In Section \ref{olprop} we study some properties of the induced metric $\sigma^\star$ and the optical length function $\ol$. Section \ref{viscsect} is devoted to a short survey about viscosity solutions in Carnot groups. In Section \ref{hlsection} we introduce the Hopf-Lax formula \eqref{hopflax} and we prove Theorem \ref{mongedir}. In Section \ref{equivsect} we show the equivalence between Monge and viscosity solutions, proving Theorem \ref{introequiv}. Finally, Section \ref{lastsect} is devoted to the proof of Theorem \ref{comparison} and Theorem \ref{stab}.
\subsection*{Acknowledgements} The authors would like to thank Andrea Pinamonti for suggesting them the study of this problem, and Eugenio Vecchi for useful discussions about these topics. The authors would also like to thank the anonymous referees  for their precise and useful corrections.
\section{Preliminaries}\label{prel}
\subsection{Carnot groups}\label{carnotgroups}
As main reference of this section, we refer the reader to \cite{BonLanUgu,MR3587666}.
A \emph{Carnot group $\G$ of dimension $n$, rank $m$ and step $k$} is a connected and simply connected Lie group whose Lie algebra $\mathfrak g$ of left-invariant vector fields is \emph{stratified}, i.e. there exist linear subspaces $\mathfrak{g}_1, \ldots,\mathfrak{g}_k$ of $\mathfrak{g}$ such that
\begin{equation}\label{stratificazione}
\mathfrak{g}=\mathfrak{g}_1\oplus \ldots \oplus \mathfrak{g}_k,\quad [\mathfrak{g}_1,\mathfrak{g}_i]=\mathfrak{g}_{i+1},\quad
\mathfrak{g}_k\neq\{0\},\quad [\mathfrak{g}_1,\mathfrak{g}_k]=\{0\},
\end{equation}
where $[\mathfrak{g}_1,\mathfrak{g}_i]$ is the subspace of ${\mathfrak{g}}$ generated by
the commutators $[X,Y]=XY-YX$ with $X\in \mathfrak{g}_1$ and $Y\in \mathfrak{g}_i$.
We denote by $k$ the \emph{step} of $\G$, by $m:=\dim(\mathfrak g_1)$ its \emph{rank} and by $n:=\dim (\mathfrak g)$ its \emph{dimension}. We say that a basis $X=(X_1,\ldots,X_n)$ of $\mathfrak g$ is \emph{adapted to the stratification} when 
\begin{equation*}
    (X_{h_{j-1}+1},\ldots, X_{h_j})\text{ is a basis of $\mathfrak g_j$ for any $j=1,\ldots,k$,}
\end{equation*}
where $h_0:=0$ and $h_j:=\sum_{i=1}^j\dim(\mathfrak g_i)$. It is in general possible to identify a Carnot group $\G$ with $\rr^n$ via its exponential map (cf. \cite[Definition 2.1.57]{BonLanUgu}). Indeed, it is possible to prove (cf. \cite[Corollary 2.2.15]{BonLanUgu}, \cite[Proposition 2.2.17]{BonLanUgu} and \cite[Proposition 2.2.18]{BonLanUgu}) that $\mathfrak g$ can be equipped with a suitable group law which realizes $\mathfrak g$ as a Carnot group, and for which the exponential map
    \begin{equation*}
        \exp:\mathfrak g\longrightarrow\G
    \end{equation*}
    is a Lie group isomorphism. To conclude, it suffices to identify $\rr^n$ with $\mathfrak g$ via \emph{exponential coordinates} of the form 
    \begin{equation*}
        (y_1,\ldots,y_n)\mapsto y_1X_1+\ldots+y_n X_n,
    \end{equation*}
    being $X_1,\ldots, X_n$ a chosen adapted basis of $\mathfrak g$. In the following, we fix an adapted basis $X_1,\ldots, X_n$ that coincides with the canonical basis of $\rr^n$ at the origin, and we adopt the notation
\begin{equation*}
    y=(y^{(1)},\ldots,y^{(k)}),
\end{equation*}
where $y^{(j)}=(y_{h_{j-1}+1}, \ldots, y_{h_{j}})$ for each $j=1,\ldots,k$. With respect to these coordinates, it is well known (cf. \cite{BonLanUgu,MR3587666}) that $0$ is the group unit of $\G$, and moreover
\begin{equation}\label{prodincoord}
    y^{-1}=(-y_1,\ldots,-y_n)\qquad\text{and}\qquad (y\cdot z)^{(1)}=y^{(1)}+z^{(1)}
\end{equation}
for any $y,z\in \G$.
For any $\lambda>0$ and any $x\in\G$, we define the \emph{left translation} $\tau_x:\mathbb{G}\longrightarrow\mathbb{G}$ and the \emph{intrinsic dilation} $\delta_\lambda:\G\longrightarrow \G$ by
\begin{equation}\label{trasledil}
    \tau_x(z) \coloneqq x \cdot z\qquad\text{and}\qquad\delta_\lambda(y):=(\lambda y^{(1)},\lambda^2y^{(2)},\ldots,\lambda^ky^{(k)})
\end{equation}
for any $y,z\in\G$.
Both $\tau_x$ and $\delta_{\lambda}$ are smooth diffeomorphisms, and $\delta_\lambda$ is a Lie group isomorphism. The subbundle of the tangent bundle $T\G$ that is spanned by the vector fields $X_1,\ldots ,X_m$ 
is called the {\em horizontal bundle} $H\G$, with fibers given by
$$ H_x\G=\mathrm{span}\left\{X_1(x),\ldots, X_m(x)\right\},$$
and the sections of $H\G$ are referred to as \emph{horizontal vector fields}. As already pointed out in the introduction, when $\Om\subseteq \G$ is an open set, the subbundle $H\Om$ of the tangent bundle $T\Om$ is defined accordingly.
In this way, a sub-Riemannian structure can be defined on $\G$ by
considering a scalar product $\langle\cdot,\cdot\rangle_x$ 
that makes
$\{ X_1, \ldots, X_m\}$ orthonormal at each point $x \in\G$. Moreover, we denote by $| \cdot |_x$ the norm
induced by $\langle\cdot,\cdot\rangle_x$, namely $| v |_x:=\sqrt{\langle v,v\rangle_x}$ for every $v\in H_x\G$. 
In the following we identify each fiber $H_x\G$ with $\rr^m$ by considering a horizontal vector field $\sum_{j=1}^ma_j(x) X_j|_x$ as the vector-valued function $x\mapsto (a_1(x),\ldots,a_m(x))$.
Notice that, with the above identification between $H_x\mathbb G$ and $\rr^m$, $\langle\cdot,\cdot\rangle_x$ coincides with Euclidean scalar product $\langle\cdot,\cdot\rangle$ on $\rr^m$.
Accordingly, we denote by $\pi$ both the smooth section defined by 
$$\pi(y) = \sum_{j = 1}^m y_j X_j(y)$$
for any $y\in\G$ and the vector valued map
\begin{equation}\label{pelaproiezincoord}
    \pi(y)=(y_1,\ldots,y_m)
\end{equation}
for any $y\in\G$. Finally, if $d$
is a distance, we denote by $B_r(x,d)$ the $d$-metric ball of radius $r>0$ centered at $x\in\G$.
\subsection{Length and geodesic distances}\label{lentandgeodistancesection}
Let us briefly recall some general facts about metric spaces for the sake of completeness. We refer to \cite{MR1835418} as main reference. Let $(M,d)$ be a possibly non-symmetric metric space. We stress that, in light of \cite[Remark 2.2.6]{MR1835418}, the statements of \cite[Chapter 2]{MR1835418} which we are going to recall hold as well in the non-symmetric setting.  If $\gamma:[0,T]\longrightarrow M$ is a continuous curve, we define its \emph{length} by
\begin{equation}\label{ldingenlenspace}
    \text{L}_d(\gamma)=\sup\left\{\sum_{j=1}^sd\left(\gamma(t_{j-1}),\gamma(t_{j})\right)\,:\,0=t_0\leq t_1\leq\ldots\leq t_{s-1}\leq t_s=T \right\}.
\end{equation}
The length functional $\text{L}_d$ is lower semicontinuous with respect to the uniform convergence of continuous curves (cf. \cite[Proposition 2.3.4]{MR1835418}, and allows to define a second distance, say $d_{\text{L}_d}$, by letting
\begin{equation}\label{intrinsicdistancegenms}
    d_{\text{L}_d}(x,y)=\inf\left\{\text{L}_d(\gamma)\,:\,\gamma:[0,1]\longrightarrow M \text{ is $d$-Lipschitz, $\gamma(0)=x$ and $\gamma(1)=y$}\right\}.
\end{equation}
Accordingly, $(M,d)$ is a \emph{length space} (cf. \cite[Definition 2.1.6]{MR1835418}) whenever
\begin{equation}\label{lengthspacedef}
    d=d_{\text{L}_d},
\end{equation}
and it is a \emph{geodesic}, or \emph{complete}, space (cf. \cite[Definition 2.1.10]{MR1835418}) whenever it is a length space such that the infimum in \eqref{intrinsicdistancegenms} is attained by a suitable $d$-Lipschitz curve. We shall refer to such curves as \emph{optimal curves}. We point out (cf. \cite[Section 2.5.2]{MR1835418}) that, if $x,y\in M$ and $\gamma:[0,1]\longrightarrow M$ is an optimal curve for $d$ connecting $x$ to $y$, then
\begin{equation}\label{onlitool}
    d(x,y)=d(x,\gamma(t))+d(\gamma(t),y)
\end{equation}
for any $t\in [0,1]$.

\subsection{Sub-Riemannian distances} 
We recall that in the sub-Riemannian setting a \emph{sub-unit curve} is a horizontal curve $\ga:[0,T]\longrightarrow\G$ such that the vector-valued function $a$ as in \eqref{curvaorizzontale} satisfies $\|a\|_{\infty}\leq 1$.
In the following, for any domain $\Om\subseteq \mathbb{G}$, we denote by $\mathcal{H}(\Om)$ the set 
\begin{equation*}
    \mathcal{H}(\Om):=\{\gamma:[0,T]\longrightarrow\Om\,:\,\gamma\text{ is sub-unit, }T>0\}.
\end{equation*}
and we define the \emph{Carnot--Carath\'{e}odory distance} on $\Om$ by
	\[ d_{\Om}(x,y) \coloneqq \inf \left\{L_{\Om}(\gamma)\colon \gamma :[0,T]\longrightarrow\Om,\,\; \gamma\in\mathcal H(\Om),\,\ga(0)=x,\,\ga(T)=y \right\},\]
 where
 \begin{equation*}
     L_\Om(\ga):=\int_0^T|\Dot\ga(t)|\,dt
 \end{equation*}
 and where by $\Dot\ga(t)$ we mean as usual the coordinates of $\Dot\ga (t)$ with respect to $X_1|_{\ga(t)},\ldots,X_m|_{\ga(t)}$.
Let us notice (cf. \cite{monti}) that an absolutely continuous curve is horizontal if and only if it is $d_\Om$-Lipschitz. Thanks to the aforementioned Chow-Rashevskii connectivity theorem, $d_\Om$ is finite for any $x,y\in\Om$. More precisely, the following crucial consequence of \cite[Proposition 1.1]{Nagel} holds.
\begin{theorem}\label{sea} Let $\mathbb{G}$ be a Carnot group of step $k$, and let $\Om\subseteq\mathbb G$ be open and connected. Then the following properties hold.
\begin{itemize}
    \item[(i)]  $d_\Om$ is a finite distance on $\Om$.
    \item[(ii)] For any domain $\tilde\Om\Subset \Om$ there exists a positive constant $C_{\tilde\Om}$ such that 
    \begin{equation*}
        C_{\tilde\Om}^{-1}|x-y|\leq d_\Om(x,y)\leq C_{\tilde\Om}|x-y|^{\frac{1}{k}}\qquad\text{ for any }x,y\in \tilde\Om.
    \end{equation*}
    \end{itemize}
\end{theorem}
 Throughout the paper, if $A, B$ are two open sets we write $A\Subset B$ meaning that $\overline A$ is compact and $\overline A\subseteq B$.  The globally defined distance $d_\G$, that is the Carnot-Carathéodory distance $d_\Om$ when $\Om=\G$, is a geodesic distance in the sense of Section \ref{lentandgeodistancesection} (cf. \cite{monti}), while $d_{\Om}$ is not geodesic in general.
The rich algebraic structure of $\G$ allows to define the well known \emph{Gauge--Koranyi distance} on $\G$ (cf. \cite{BonLanUgu}). To this aim, consider the function
    \[\|(y^{(1)},\ldots,y^{(k)})\|=\Bigg(\sum_{j=1}^k|y^{(j)}|^{\frac{2k!}{j}}\Bigg)^\frac{1}{2k!}\]
    for any $y\in\G$. This is a \emph{homogeneous norm} (cf. \cite{BonLanUgu}) on $\G$,  and has the remarkable advantage of being smooth outside the origin (cf. \cite[Example 5.1.2]{BonLanUgu}).  It induces the homogeneous distance 
    \begin{equation}\label{od}
        \d(y,z)=\|y^{-1}\cdot z\|
    \end{equation}
    for any $y,z\in\G$. It is well known (cf. \cite{Nagel}) that $\d$ and
    $d_{\mathbb G}$ are equivalent distances on $\mathbb G$. Moreover, if $\Om\subseteq \mathbb G$ is any domain, then clearly $d_{\mathbb G}\leq d_\Om$. Thanks to Theorem \ref{sea}, the following holds.

    \begin{proposition}\label{loceq}
        Let $\mathbb{G}$ be a Carnot group of step $k$, and let $\Om\subseteq\mathbb G$ be open and connected. Then $d_\Om$, $d_{\mathbb G}$ and $\d$ are locally equivalent on $\Om$.
    \end{proposition}
    To prove Proposition \ref{loceq}, we need the following result.
    \begin{lemma}\label{optdg}
        Let $\mathbb{G}$ be a Carnot group of step $k$, and let $\Om\subseteq\mathbb G$ be open and connected. Then, for any $x_0\in\Om$, there exists $r>0$ such that, for any $x,y\in B_r(x_0,d_{\mathbb G})$, any optimal curve for $d_{\mathbb G}(x,y)$ lies in $\Om$.
    \end{lemma}
    \begin{proof}
        Assume by contradiction that there exists $x_0\in\Om$ and sequences $(x_h)_h,(y_h)_h$, $(\ga_h)_h$ such that $d_{\mathbb G}(x_0,x_h),d_{\mathbb G}(x_0,y_h)<\frac{1}{h}$, $\gamma_h:[0,T_h]\longrightarrow \mathbb G$ is sub-unit and is optimal for $d_{\mathbb G}(x_h,y_h)$, and there exists $0<t_h<T_h$ such that $z_h:=\ga_h(t_h)\in\partial\Om$. Up to a subsequence, there exists $R>0$ such that $(x_h)_h,(y_h)_h\subseteq B_R(x_0,d_\G)\Subset\Om$. Set
        \begin{equation*}
            D=\inf\{d_\G(z,w)\,:\,z\in\partial\Om,\;w\in\partial B_R(x_0,d_\G)\}.
        \end{equation*}
         Since $B_R(x_0,d_\G)\Subset\Om$, then $D>0$. On one hand
        \begin{equation*}
            d_{\mathbb G}(x_h,y_h)\leq  d_\G(x_h,x_0)+d_\G(x_0,y_h)\to 0
        \end{equation*}
    as $h\to\infty$. On the other hand, in view of the choice of $\gamma_h$  and \eqref{onlitool}, 
    \begin{equation*}
        d_\G(x_h,y_h)=d_\G(x_h,z_h)+ d_\G(z_h,y_h)\geq 2D>0.
    \end{equation*}
    A contradiction then follows.
    \end{proof}
    \begin{proof}[Proof of Proposition \ref{loceq}]
         As already pointed out, by definition it clearly follows that $d_\Om\geq d_\G$.  Therefore, we are left to show that for any domain $\tilde\Om\Subset\Om$ there exists $K_{\tilde\Om}>0$ such that $d_{\mathbb G}\ge\ K_{\tilde\Om}d_\Om$. Assume by contradiction that there exists a domain $\tilde\Om\Subset\Om$ and two sequences $(x_h)_h,(y_h)_h\subseteq\tilde\Om $ such that
        $$d_{\mathbb G}(x_h,y_h)<\frac{1}{h}d_\Om(x_h,y_h)$$ for any $h \in \mathbb{N}$. Let $D$ be the Euclidean diameter of $\tilde\Om$. Since $\tilde\Om$ is bounded, then $D<\infty$. Thanks to Theorem \ref{sea}, we have that
        \begin{equation*}
            d_\G(x_h,y_h)<\frac{1}{h}d_\Om(x_h,y_h)\leq\frac{1}{h}\sup_{x,y\in\tilde\Om}d_\Om(x,y)\leq \frac{C_{\tilde\Om}}{h}\sup_{x,y\in\tilde\Om}|x-y|^{\frac{1}{k}}\leq \frac{C_{\tilde\Om}D^\frac{1}{k}}{h}.
        \end{equation*}
        This implies that $d_{\mathbb G}(x_h,y_h)\to 0$. Therefore, up to a subsequence, we can assume that $x_h,y_h\to x_0$ for some $x_0\in\Om$. Choose $r$ as in Lemma \ref{optdg}, and assume up to a subsequence that $(x_h)_h,(y_h)_h\subseteq B_r(x_0,d_\G)$. Then Lemma \ref{optdg} implies that 
        \begin{equation*}
            d_{\mathbb G}(x_h,y_h)=d_\Om(x_h,y_h),
        \end{equation*}
        a contradiction. 
    \end{proof}
\subsection{Calculus on Carnot groups}\label{calculus}
Given $u\in L^1_{\loc}(\Om)$, we define its distributional \emph{horizontal gradient} $Xu$ by
\begin{equation*}
    Xu(\varphi):=-\int_\Om u\sum_{j=1}^mX_j\varphi_j\,dx
\end{equation*}
for any $\varphi=(\varphi_1,\ldots,\varphi_m)\in C^\infty_c(\Om,\rr^m)$ (cf. \cite{MR3587666}). This notion allows to define in the obvious way the classical functional spaces $W^{1,\infty}_X(\Om)$, $W^{1,\infty}_{X,\loc}(\Om)$ and $C^1_X(\Om)$. More specifically, we say that $u\in W^{1,\infty}_X(\Om)$ if $u\in L^\infty(\Om)$ and $Xu\in L^\infty(\Om,\rr^m)$, and that $u\in C^1_X(\Om)$ if $u$ is continuous and $Xu$ is continuous. The space $W^{1,\infty}_{X,\loc}(\Om)$ is defined accordingly. If $d$ is a distance, we define the space
\begin{equation*}
    \lip(\Om,d):=\left\{u\in C(\Om)\,:\,\text{there exists $C>0$ such that }\sup_{x\neq y\in\Om}\frac{|u(x)-u(y)|}{d(x,y)}\leq C\right\}.
\end{equation*}
The space $\lip_{\loc}(\Om,d)$ is defined in the obvious way.
It is well known (cf. \cite{GN}) that
\begin{equation*}
    W^{1,\infty}_{X,\loc}(\Omega)=\lip_{\loc}(\Om,d_\Om).
\end{equation*}
We conclude this section recalling the following differentiability result due to Pansu (cf. \cite{Pansu}). 
\begin{theorem}\label{Pansu}
Let $\Omega\subset \mathbb{G}$ be an open set. Let $u\in W^{1,\infty}_{X,\loc}(\Om)$. Then $u$ is \emph{Pansu-differentiable} at almost every $x_0 \in \Omega$, that is 
\begin{equation*}
    \lim_{x\to x_0}\frac{u(x)-u(x_0)-\langle Xu(x_0),\pi(x_0^{-1}\cdot x)\rangle}{d_\Om(x_0,x)}=0
\end{equation*}
for almost every $x_0\in\Om$.
\end{theorem}
\subsection{Subgradient in Carnot groups}
In this section we recall some properties of the so-called \emph{$(X,N)$-subgradient} of a function $u\in W_{X,\emph{\loc}}^{1,\infty}(\Om)$, introduced in \cite{PVW} as a generalization of the classical Clarke's subdifferential (cf. \cite{clarke}) and defined by
\begin{equation*}
	\sub (x):=\overline{co}\left\{\lim_{n\to\infty}Xu(y_n)\,: \,y_n\to x,\,y_n\notin N\text{ and }\lim_{n\to\infty}Xu(y_n) \text{ exists}\right\}
\end{equation*}
for any $x\in \Om$, where $N\subseteq\Om$ is any Lebesgue negligible set containing the non-Lebesgue points of $Xu$ and $\overline{co}
$ denotes the closure of the convex hull. In the sequel we will need the following result, which can be found as \cite[Proposition 2.5]{PVW}.
\begin{proposition}\label{w2}
	Let  $u\in W^{1,\infty}_{X,\loc}(\Om)$ and let $\gamma:[ 0,T ]\longrightarrow\Om$ be a horizontal curve as in \eqref{curvaorizzontale}.
	The function $t\mapsto u(\gamma(t))$ belongs to $W^{1,\infty}( 0,T )$, and there exists a function $\vartheta\in L^\infty(( 0,T ),\rr^m)$ such that 
	\begin{equation*}
		\frac{d (u \circ \gamma)(t)}{d t}= \langle \vartheta(t), a(t) \rangle
	\end{equation*}
	for a.e. $t\in ( 0,T )$. Moreover
	\begin{equation*}
		\vartheta(t)\in\sub (\gamma (t))
	\end{equation*}
	for a.e. $t\in ( 0,T ).$
\end{proposition}
\section{Some properties of $\sigma^\star$ and $\ol$}\label{olprop}
Here and in the following we will be focused on Hamilton--Jacobi equations as in \eqref{eq:HJI}, that is
\begin{equation}\tag{H-J}
\label{eq:HJ}
H (x,X u)=0
\end{equation}
on $\Om$, where $\Om$ is a subdomain of $\G$ and $H$ satisfies the structural assumptions (H).
Since the notion of Monge solution heavily depends on the properties of the associated optical length function, and hence on the properties of $\sigma^\star$, let us make some preliminary considerations on these objects.
First, notice that condition $(\text{H}_3)$ is equivalent to the estimate
\begin{equation}
\label{equistima}
    \frac{1}{\alpha} |v|_x \leq \sigma^\star(x, v) \leq \alpha |v|_x \quad \mbox{for every } (x, v) \in H \mathbb{G}.
\end{equation}
Moreover the following simple result, which is the sub-Riemannian analogous of \cite[Lemma 4.2]{monge}, will be useful to state the equivalence between Monge and viscosity solutions in the continuous setting. We refer to \cite{esspinpas} for an account of sub-Finsler metrics.
\begin{lemma}
\label{lm:subFinul}
 Let $\Om\subseteq \G$ be an open set, and let $H\Om$ be the horizontal bundle defined in Section \ref{carnotgroups}. 
$\sigma^\star: H \Om \longrightarrow \mathbb{R} $ is a sub-Finsler convex metric.
Moreover, for any $v\in\rr^m$, the following hold.
\begin{itemize}
\item[(i)] If $H$ is upper semicontinuous on $H \Om$, then $\sigma^\star (\cdot, v)$ is lower semicontinuous on $\Om$.
\item[(ii)]  If $H$ is lower semicontinuous on $H \Om$, then $\sigma^\star(\cdot,v)$ is upper semicontinuous on $\Om$.
\end{itemize}
\end{lemma}
Regarding the optical length function, an easy computation shows that
\begin{equation}\label{induceddd2}
    \ol(x,y)=\inf \left  \{ \int_0^T \sigma^\star(\ga(t),\dot{\ga}(t)) \,dt    \ :\,\ga:[0,T]\longrightarrow\Om,\,\ga\in\mathcal{H}(\Om),\,\ga(0)=x\, ,\ga(T)=y \right\}
\end{equation}
for any $x,y\in\Om$.
The quantity \eqref{induceddd2} is well-defined, both because the map $t \mapsto \sigma^\star(\gamma(t), \dot{\gamma}(t))$ is Borel measurable on the horizontal bundle, and because, as already mentioned, every two points in $\Om$ can be connected by a horizontal curve.
However, $\ol$ can presents some pathological behaviour without some semicontinuity assumptions (see \cite[Example 5.5]{esspinpas}).
Let us discuss some properties of $d_{\sigma^\star}$
which will be useful in the sequel. 
\begin{lemma}\label{prop:optimalpath}
 Let $\alpha$ be as in $(\text{H}_3)$. The following properties hold.
\begin{itemize}
    \item [$(i)$] $\ol$ is a non-symmetric distance on $\Om$.
    \item[$(ii)$] $\ol$ is equivalent to $d_\Om$ on $\Om$, i.e.
    \begin{equation*}
        \frac{1}{\alpha} \, d_{\Om} (x,y) \le d_{\sigma^\star}(x,y) \le \alpha \, d_{\Om}(x,y)
    \end{equation*}
     for any $x,y\in\Om$.
    \item[$(iii)$] $\ol$ is $d_\Om$-Lipschitz on $\Om\times\Om$, that is
    \begin{equation*}
        |\ol(x,y)-\ol(z,w)|\leq\alpha(d_\Om(x,z)+d_\Om(y,w))
    \end{equation*}
    for any $x,y,z,w\in\Om$.
\end{itemize}
\end{lemma}
\begin{proof}
The proof of $(i)$ follows as in \cite[Lemma 5.7]{esspinpas}. $(ii)$ is an easy consequence of estimate \eqref{equistima}. Let us show $(iii)$. To this aim, fix $x,y,z,w\in\Om$. Being $\ol$ a distance and thanks to point $(ii)$, we have that  
\begin{equation*}
\begin{split}
     \ol(x,y)-\ol (z,w)&=\ol(x,y)-\ol(z,y)+\ol(z,y)-\ol(z,w)\\
&\leq\ol(x,z)+\ol(w,y)\\
&\leq\alpha(d_\Om(x,z)+d_\Om(y,w))
\end{split}
\end{equation*}
and,  exchanging the roles of $(z,w)$ and $(x,y)$,
\begin{equation*}
     \ol (z,w)-\ol (x,y)\leq\alpha(d_\Om(x,z)+d_\Om(y,w)).
\end{equation*}
\end{proof}
Therefore $(\Om,\ol)$ is a non-symmetric metric space.  In view of general results in metric spaces (cf. \cite[Proposition 2.4.1]{MR1835418}, $\ol$ is a length distance in the sense of Section \ref{lentandgeodistancesection}.  However, we already know that it is not geodesic in general, since, for instance, $(\Om,d_\Om)$ may not be geodesic. Nevertheless, exploiting standard arguments of analysis in metric spaces (cf. \cite{ambrosio}) it can be shown that $(\Om,\ol)$ is locally geodesic in the following sense. 
\begin{proposition}\label{locgeo}
For any $x_0\in\Om$ there exists $r>0$ such that for any $x,y\in B_r(x_0,d_\Om)$ there exists a horizontal curve  $\ga:[0,1]\longrightarrow\Om$ such that $\ga(0)=x$, $\ga(1)=y$ and 
  $$
        \ol(x,y)=\emph{L}_{\ol}(\gamma),
$$
where $\emph{L}_{\ol}$ is as in \eqref{ldingenlenspace}.
\end{proposition}
We first need the following technical lemma,  whose proof is omitted being analogous to the proof of the forthcoming Lemma \ref{equivsigma}. 
\begin{lemma}\label{distaloc}
    For any $x_0\in\Om$, and for any $R>0$ such that $B_R(x_0,d_\Om)\Subset\Om$, there exists $0<r<R$ and $\bar\varepsilon>0$ such that, for any $x,y\in B_r(x_0,d_\Om)$ and for any $0<\varepsilon<\bar\varepsilon$, every horizontal curve $\gamma:[0,1]\longrightarrow \Om$ such that $\gamma(0)=x$, $\gamma(1)=y$ and 
    \begin{equation*}
        \ol(x,y)\geq { \emph{L}_{d_{\sigma*}}}(\gamma)-\varepsilon
    \end{equation*}
     lies in $B_R(x_0,d_\Om)$.
\end{lemma}
\begin{proof}[Proof of Proposition \ref{locgeo}]
Let $x_0\in\Om$ and $R>0$ be such that $B_R(x_0,d_\Om)\Subset\Om$. Then let $r>0$ be as in Lemma \ref{distaloc}. Let $x, y\in B_r(x_0,d_\Om)$ and let $(\gamma_h)_h$ be a sequence of horizontal curves such that $\ga_h(0)=x$, $\ga_h(1)=y$ and 
    \begin{equation}\label{lsc}
        {\text{L}_{d_{\sigma*}}}(\ga_h)\leq \ol (x,y)+\frac{1}{h}.
    \end{equation}
    In view of Lemma \ref{distaloc}, we can assume that $\ga_h([0,1])\subseteq\overline{B_R(x_0,d_\Om)}\subseteq\Om$ for any $h\in\mathbb N$.  Clearly $(\overline{B_R(x_0,d_\Om)},\ol)$ is a compact metric space and the sequence $(\ga_h)_{h\in \nn}$ is uniformly bounded. Arguing \emph{verbatim} as in the proof of \cite[Theorem 4.3.2]{ambrosio},  $(\ga_h)_{h\in \nn}$ is also equicontinuous with respect to $\ol$. Therefore
    Ascoli--Arzelà's Theorem implies the existence of a horizontal curve $\gamma:[0,1]\longrightarrow \Om$ such that $(\ga_h)_h$ converges uniformly to $\ga$. 
   In particular, $\gamma(0)=x$ and $\gamma(1)=y$.
   Hence, combining \eqref{lengthspacedef} and  \eqref{lsc}  with the lower semicontinuity of $\text{L}_{d_{\sigma^\star}}$ (cf. Section \ref{lentandgeodistancesection}) we infer that
     \begin{equation*}
         \ol (x,y) \leq \text{L}_{d_{\sigma^\star}}(\gamma)\leq\liminf_{h\to\infty}\text{L}_{d_{\sigma^\star}}(\gamma_h) \leq \liminf_{h\to\infty}\left ( \ol (x,y)+\frac{1}{h}\right)=\ol (x,y).
     \end{equation*}
     Therefore we conclude that $\gamma$ is optimal for $\ol(x,y)$, and the thesis follows.
\end{proof}
\begin{proposition}\label{stimapermongeimplicavisc}
    Assume that $H$ is upper semicontinuous on $H \Om$.
    Then it holds that
    \begin{equation*}
        \liminf_{t\to 0^+}\frac{d_{\sigma^\star}(x,x\cdot\delta_t(\xi,\eta))}{t}\geq \sigma^\star (x,\xi)
    \end{equation*}
    for any $x\in\Om$, $\xi\in\rr^m$ and $\eta\in \rr^{n-m}$.
    \end{proposition}
    \begin{proof}
        Let us fix $x\in\Om$, $\xi\in\rr^m$ and $\eta\in\rr^{n-m}$. Since $H$ is upper semicontinuous on $H \Om$, then $\sigma^\star(\cdot,\xi)$ is lower semicontinuous on $\Om$ by Lemma \ref{lm:subFinul}. This is equivalent to say that, for any $\varepsilon>0$ and for any $\tilde\xi\in \mathbb{S}^{m-1}$, there exists $r=r(x,\epsilon,\tilde \xi)$ such that $\sigma^\star(y,\tilde\xi)\geq\sigma^\star(x,\tilde\xi)-\varepsilon$ for any $y\in B_r(x,d_\Om)$. Recalling that $\sigma^\star$ is Lipschitz in the second entry and exploiting a standard compactness argument  (cf. the proof of \cite[Proposition 2.9]{monge}),  we infer that for any $\varepsilon>0$ there exists $r=r(x,\varepsilon)>0$ such that 
        \begin{equation}\label{lscnf}
            \sigma^\star(y,\tilde\xi)\geq\sigma^\star(x,\tilde \xi)-\varepsilon
        \end{equation}
        for any $y\in B_r(x,d_\Om)$ and any $\tilde \xi\in \mathbb{S}^{m-1}$. Let us choose a sequence of sub-unit curves $\ga_h:[0,t_h]\longrightarrow\Om$ in such a way that $\gamma_h(0)=x$, $\gamma_h(t_h)=x\cdot\delta_{t_h}(\xi,\eta)$ and
        \begin{equation*}
            \liminf_{t\to 0^+}\frac{d_{\sigma^\star}(x,x\cdot\delta_t(\xi,\eta))}{t}=\liminf_{h\to\infty}\ave_0^{t_h}\sigma^\star(\ga_h(t),\Dot\ga_h(t))\,dt.
        \end{equation*}
        Since $\lim_{t\to 0^+}x\cdot\delta_t(\xi,\eta)=x$, and in view of Lemma \ref{distaloc}, the sequence of curves can be choosen in such a way that $\gamma_h([0,t_h])\subseteq B_r(x,d_\Om)$ for any $h\in\mathbb N$. Therefore, exploiting \eqref{lscnf}, we infer that
        \begin{equation*}
            \liminf_{h\to\infty}\ave_0^{t_h}\sigma^\star(\ga_h(t),\Dot\ga_h(t))\,dt\geq \liminf_{h\to\infty}\ave_0^{t_h}\sigma^\star(x,\Dot\ga_h(t))\,dt-\varepsilon.
        \end{equation*}
For any $h\in\mathbb N$, set $\gamma_h=(\gamma_h^1,\ldots,\gamma_h^m,\gamma_h^{m+1},\ldots,\gamma_h^n)$. We recall that in the previous equations $\Dot\ga_h$ is the $m$-tuple of the components of $\Dot\ga_h$ along the generating vector fields. In other words, we mean $\Dot\ga_h(t)=(a_h^1(t),\ldots,a_h^m(t))$, where $\Dot\ga_h(t)=\sum_{j=1}^ma_h^j(t)X_j(\ga_h(t))$. It is then easy to see that $\Dot\ga_h^j=a_h^j$ for any $j=1,\ldots, m$. Therefore, 
 by \eqref{prodincoord}, \eqref{trasledil} and the fundamental theorem of calculus for absolutely continuous functions, we infer that
\begin{equation}\label{tfcexplanation}
\ave_0^{t_h}\Dot\ga_h(t)\,dt=\frac{\pi(\gamma_h(t_h))-\pi(\gamma(0))}{t_h}=\xi
\end{equation}
for any $h\in\mathbb N$, where $\pi$ is the projection map defined in \eqref{pelaproiezincoord}. Combining \eqref{tfcexplanation} with
the convexity properties of $\sigma^\star$ and Jensen's inequality, we get that 
\begin{equation*}
\liminf_{h\to\infty}\ave_0^{t_h}\sigma^\star(x,\Dot\ga_h(t))\,dt-\varepsilon\geq \liminf_{h\to\infty}\sigma^\star\left(x,\ave_0^{t_h}\Dot\ga_h(t)\,dt\right)-\varepsilon=\sigma^\star(x,\xi)-\varepsilon.
\end{equation*}
The thesis follows letting $\varepsilon$ go to $0$.
    \end{proof}
\section{Viscosity solutions for continuous Hamilton-Jacobi equations}\label{viscsect}
When the Hamiltonian $H$ is continuous, the study of \eqref{eq:HJ} can be carried out in the setting of sub-Riemannian viscosity solutions. To introduce this notion, we recall that 
the \emph{first order superjet} of $u\in C(\Om)$ at a point $x_0\in\Om$ is defined by
\begin{equation*}
    \partial_X^{+} u(x_0)=\{v\in\rr^m\,:\,u(x) \leq u(x_0) + \langle v, \pi(x_0^{-1} \cdot x) \rangle + o(d_{\Om}( x_0,x))\},
\end{equation*}
while the \emph{first order subjet} of $u$ at $x_0$ is defined by
\begin{equation*}
    \partial_X^{-} u(x_0)=\{v\in\rr^m\,:\,u(x) \geq u(x_0) + \langle v, \pi(x_0^{-1} \cdot x) \rangle + o(d_\Om( x_0,x))\}.
\end{equation*}
It is easy to see that $\partial_X^{+} u(x_0)$ and $\partial_X^{-} u(x_0)$ are closed and convex, and that they may be empty in general. Moreover, in view of Proposition \ref{loceq}, in the previous definition $d_\Om$ can be equivalently replaced by $\d$ or $d_{\mathbb G}$.
In the Euclidean setting (cf. \cite{users}) the notion of viscosity solution can be equivalently given exploiting either jets or suitable test functions (cf. \cite{CL,CraEvaLi}). Following this path (cf. \cite{StrofMan,Strof})
we say that a function $u \in C(\Om)$ is a 
\emph{jet subsolution} to \eqref{eq:HJ} in $\Om$ if 
\[
H(x_0, v) \leq 0 \]
for every $x_0 \in \Om$ and every $ v \in \partial_X^+ u(x_0)$. Similarly, $u$ is a \emph{jet supersolution}
to \eqref{eq:HJ} in $\Om$ if
\[
H(x_0, v) \geq 0\]
for every $x_0 \in \Om$ and every $ v \in \partial_X^- u(x_0)$. Finally, $u$ is a \emph{jet solution} to \eqref{eq:HJ} if it is both a jet subsolution and a jet supersolution. On the other hand, we say that $u$ is a \emph{viscosity subsolution} to \eqref{eq:HJ}
if $$ H(x_0, X \psi(x_0)) \leq 0$$ for any $x_0 \in \Omega$ and for any $\psi \in C_X^1(\Omega)$ such that
\[
u(x_0) - \psi(x_0) \geq u(x) - \psi(x)
\]
for any $x$ in a neighborhood of $x_0$. We say that $u$ is a \emph{viscosity supersolution} to \eqref{eq:HJ} if $$ H(x_0, X \psi(x_0)) \geq 0$$ for any $x_0 \in \Omega$ and for any $\psi \in C_X^1(\Omega)$ such that
\[
u(x_0) - \psi(x_0) \leq u(x) - \psi(x),
\]
for any $x$ in a neighborhood of $x_0$. Again, we say that $u$ is a \emph{viscosity solution} to \eqref{eq:HJ} if it is both a viscosity subsolution and a viscosity supersolution. The next proposition shows that even in this sub-Riemannian setting these two definitions are equivalent.
The following proof is inspired by \cite{StrofMan}.

\begin{proposition}\label{jetvisco}
    Let $\Om\subseteq\G$ be open. Assume that $H$ is continuous. Then $u\in C(\Om)$ is a jet subsolution (resp. supersolution) to \eqref{eq:HJ} if and only if it is a viscosity subsolution (resp. supersolution) to \eqref{eq:HJ}.
\end{proposition}
\begin{proof}
    We prove only the half of the claim concerning subsolutions, being the other half analogous. The fact that a jet subsolution is a viscosity subsolution follows from \cite[Proposition 3.2]{CGPV}. On the contrary, assume that $u$ is a viscosity subsolution to \eqref{eq:HJ}, let $x_0\in\Om$ and $p\in\partial_X^+u(x_0)$. Let $\d$ be as in \eqref{od}. It is well known that  $y\mapsto \d(x_0,y)$ is smooth outside $x_0$ and its horizontal gradient is bounded near $x_0$. Since $p\in\partial_X^+u(x_0)$, then
\begin{equation}\label{gaugesuperdiff}
        u(x) \leq u(x_0) + \langle p, \pi(x_0^{-1} \cdot x) \rangle + o(\d( x_0,x)).
    \end{equation}
Let $R>0$ be such that $B_R(x_0,\d)\Subset\Om$, and define $g:(0,R]\longrightarrow\mathbb R$ by
\begin{equation*}
    g(r) := \sup_{x\in B_r(x_0,\d)}\frac{\max \{ 0,u(x)-u(x_0)- \langle p,\pi(x_0^{-1} \cdot x) \rangle \}}{\d( x_0,x)}.
\end{equation*}
Then $g$ is nondecreasing and, by the choice of $p$, $\lim_{r\to 0}g(r)=0$, Hence there exists $\tilde g\in C([0,R])$ such that $\tilde g$ is nondecreasing, $\tilde g(0)=0$ and $\tilde g\geq g$. Let $G(r):=\int _0^r\tilde g(\tau)d\tau$. Then $G\in C^1([0,R[)$ and $G(0)=G'(0)=0$. Moreover, for any $0<r<\frac{R}{2}$, it holds that
\begin{equation}\label{stimavisco}
    G(2r)\geq \int _r^{2r}\tilde g(\tau)d\tau\geq r\tilde g(r)\geq rg(r).
\end{equation}
Let us define $\varphi(x)=u(x_0)+\langle p, \pi(x_0^{-1} \cdot x) \rangle+G(2\d(x,x_0))$. Then 
$\varphi\in C^1_X(B_\frac{R}{2}(x_0,\d))$, $u(x_0)=\varphi(x_0)$ and $X \varphi(x_0)=p$. Finally, notice that \eqref{stimavisco} and the definition of $g$ imply that $u(x)\leq\varphi(x)$ on $B_\frac{R}{2}(x_0,\d)$.
Therefore, being $u$ a viscosity subsolution, we conclude that
\begin{equation*}
    H(x_0,p)=H(x_0,X\varphi(x_0))\leq 0.
\end{equation*}
    \end{proof}
Moreover, when $H$ enjoys the following mild convexity properties in its second entry, a locally Lipschitz function is a viscosity subsolution if and only if it satisfies \eqref{eq:HJ} pointwise almost everywhere. 
\begin{proposition}\label{equivsol}
    Let $\Om$ be an open subset of $\mathbb G$. Assume that $H$ is continuous and that $Z(x)$ is convex for any $x\in\Om$. Let $u\in W^{1,\infty}_{X,\loc}(\Om)$. Then the following conditions are equivalent.
    \begin{itemize}
        \item [$(i)$] $u$ is a viscosity subsolution to \eqref{eq:HJ}.
        \item[$(ii)$] $u$ is a jet subsolution to \eqref{eq:HJ}.
        \item[$(iii)$] $H(x, X u(x))\leq 0$ for almost every $x\in\Om$.
    \end{itemize}
\end{proposition}
\begin{proof}
    The implication $(i)\iff(ii)$ follows from Proposition \ref{jetvisco}. Moreover, $(iii)\implies (i)$ follows from \cite[Theorem 3.7]{CGPV}. Finally, we prove $(ii)\implies (iii)$. Let $x\in\Om$ be such that $u$ is Pansu-differentiable at $x$. Then clearly $Xu(x)\in\partial_X^+u(x)$, and so $H(x,X u(x))\leq 0$.
\end{proof}
To conclude this section, we point out that the sub-Riemannian Hamilton--Jacobi equation \eqref{eq:HJ} can be viewed as an Euclidean equation in the following sense. Let $C(x)$ denote the $m\times n$ matrix whose rows correspond to the coefficients of the generating vector fields of $\mathfrak g_1$ at $x$. 
We define the auxiliary Hamiltonian $\Tilde{H}: \Omega \times \mathbb{R}^n \longrightarrow \mathbb{R}$ by
\begin{equation}\label{eqH2}
\Tilde{H}(x, v) = H(x, v\cdot C(x)^T) 
\end{equation}
for any $(x,v) \in \Omega \times \mathbb{R}^n$.
It is easy to see that $\Tilde{H}\in C(\Om\times\mathbb R^n)$ when $H$ is continuous. With the next result, we show that sub-Riemannian viscosity solutions to \eqref{eq:HJ} coincides with Euclidean viscosity solutions to the Hamilton--Jacobi equation associated to \eqref{eqH2}.

\begin{proposition}\label{euclivscarnot}
    Let $\Om$ be an open subset of $\mathbb G$. Let $\Tilde{H}$ be as in \eqref{eqH2}. Then $u\in C(\Om)$ is a viscosity subsolution (resp. supersolution) to
    \begin{equation}\label{subriem}
        H(x, X u)=0
    \end{equation}
    if and only if $u$ is a viscosity subsolution (resp. supersolution) to 
    \begin{equation}\label{euclid}
        \Tilde{H}(x,\nabla u)=0.
    \end{equation}
\end{proposition}
\begin{proof}
Since $C^1(\Om)\subseteq C^1_X(\Om)$, then a viscosity solution to \eqref{subriem} is a viscosity solution to \eqref{euclid}. To prove the converse implication we only show that viscosity subsolutions to \eqref{euclid} are viscosity subsolutions to \eqref{subriem}, being the other part of the proof analogous. Therefore, assume that $u$ is a viscosity subsolution to \eqref{euclid}, let $x_0\in\Om$ and let $\varphi\in C^1_X(\Om)$ be such that $u(x_0)=\varphi(x_0)$ and $\varphi(x)>u(x)$ for any $x\in B_{2r}(x_0,d_\Om)$, for some $r>0$ small enough to ensure that $B_{2r}(x_0,d_\Om )\Subset\Om$. Thanks to \cite[Proposition 2.4]{CGPV} (cf. also \cite[Proposition 1.20]{FS}), there exists a sequence $(\varphi_h)_h\subseteq C^\infty(\Om)$ converging to $\varphi$ in $C^1_X(B_{2r}(x_0,d_\Om))$. For any $h\in\mathbb N$, let $x_h$ be a maximum point for $u-\varphi_h$ in $\overline{B_r(x_0,d_\Om)}$. We claim that $x_h\to x_0$ as $h\to +\infty$. Otherwise, we can assume that, up to a subsequence, $x_h\to x_1$ for some $x_1\neq x_0$ such that $x_1\in \overline{B_r(x_0,d_\Om)}$. Recalling that $u(x_h)-\varphi_h(x_h)\geq u(x_0)-\varphi_h(x_0)$ for any $h\in\mathbb N$, and since $x_h\to x_1$ and $\varphi_h\to\varphi$ uniformly on $B_{2r}(x_0,d_\Om)$, we pass to the limit and we infer that $u(x_1)-\varphi(x_1)\geq u(x_0)-\varphi(x_0)=0.$ Therefore $\varphi(x_1)\leq u(x_1)$, a contradiction. By our choice of $x_h$, and thanks to \eqref{euclid}, we get that
\begin{equation*}
    H(x_h,X\varphi_h(x_h))=\Tilde{H}(x_h,\nabla \varphi_h(x_h))\leq 0.
\end{equation*}
Therefore, since $H$ is continuous, $x_h\to x$ and $X \varphi_h\to X \varphi$ uniformly on $B_{2r}(x_0,d_\Om)$, passing to the limit in the previous inequality we conclude that 
\begin{equation*}
    H(x_0, X \varphi(x_0))\leq 0.
\end{equation*}
Hence $u$ is a viscosity subsolution to \eqref{subriem}.
\end{proof}
\section{A Sub-Riemannian Hopf--Lax Formula for the Dirichlet Problem}\label{hlsection}
As already mentioned, the properties of Monge subsolutions and supersolutions strictly depend
on those enjoyed by the optical length function $\ol$. Moreover, as it happens in the viscosity setting, $d_\Om$ can be equivalently replaced by $\d$ or $d_{\mathbb G}$.
Now we explain how to replace $\ol$ with suitable extensions as already explained in the introduction.  According to the latter, we recall the family defined in \eqref{kzero}, that is
\begin{equation*}
    \mathcal K(H,\Om):=\{K:\rr^n\times\rr^m\longrightarrow \rr\,:\,\text{$K$ satisfies {\rm (H)} and $K\equiv H$ on $\Om\times\rr^m$}\}.
\end{equation*}
As already mentioned, $\mathcal K(H,\Om)$ is always non-empty. For instance, it is always the case that 
 \begin{equation}\label{numeraperdopo}
        K(x,\xi):=
        \begin{cases}
            H(x,\xi)\qquad\text{if }(x,\xi)\in\Om\times\rr^m\\
            |\xi|-\alpha\qquad\text{otherwise}
        \end{cases}
    \end{equation}
    belongs to  $\mathcal K(H,\Om)$.
 For a fixed $K\in\mathcal K(H,\Om)$, we can consider the associated $\sk$ and $\olk$. We want to show that the notion of Monge solution is independent of the choice of $K\in\mathcal K(H,\Om)$. To this aim, we prove the following preliminary result.
\begin{lemma}\label{equivsigma}
    For any $K\in\mathcal K(H,\Om)$, for any $x_0\in\Om$ and for any $R>0$ such that $B_R(x_0,d_\Om)\subseteq \Om$ there exists $r>0$ and $\bar\varepsilon>0$ such that, for any $x\in B_r(x_0,d_\Om)$ and for any $0<\varepsilon<\bar\varepsilon$, any curve $\gamma:[0,T]\longrightarrow\rr^n$ such that $\gamma$ is sub-unit, $\gamma(0)=x_0$, $\ga(T)=x$ and
    \begin{equation*}
        \olk(x_0,x)\geq  \int_0^T \sigma_K^\star(\ga(t),\dot{\ga}(t)) \ dt-\varepsilon
    \end{equation*}
    lies in $B_R(x_0,d_\Om)$.
\end{lemma}
\begin{proof}
    Assume by contradiction that there exists $K:\rr^n\times\rr^m\longrightarrow\rr$ such that $K\in\mathcal K(H,\Om)$, $x_0\in\Om$, $R>0$ with $B_R(x_0)\subseteq\Om$ and sequences $(x_h)_h$ and $(\gamma_h)_h\subseteq\mathcal H(\rr^n)$, with $\ga_h:[0,T_h]\longrightarrow \rr^n$, $\ga_h(0)=x_0$, $\ga_h(T_h)=x_h$ and
    \begin{equation*}
        \olk(x_0,x_h)\geq \int_0^{T_h} \sigma_K^\star(\ga_h(t),\dot{\ga_h}(t)) \ dt-\frac{1}{h}    
    \end{equation*}
    such that $x_h\to x_0$ and for any $h$ there exists $0<t_h<T_h$ such that $z_h:=\ga(t_h)\in\partial B_R(x_0,d_\Om)$. Since $K$ satisfies {\rm (H)} on $\rr^n$, it follows that
    \begin{equation*}
        \olk(x_0,x_h)\leq \alpha d_{\G}(x_0,x_h)\leq \alpha d_\Om(x_0,x_h)\to 0
    \end{equation*}
    as $h\to\infty$. On the other hand, in view of the choice of $\gamma_h$ and Proposition \ref{loceq}, there exists $C>0$ such that
    \begin{equation*}
    \begin{split}
\olk(x_0,x_h)&\geq \olk(x_0,z_h)+\olk(z_h,x_h)-\frac{1}{h}\\
&\geq\frac{1}{\alpha}d_{\mathbb G}(x_0,z_h)-\frac{1}{h}\\
&\geq\frac{C}{\alpha}d_{\Om}(x_0,z_h)-\frac{1}{h}\\
&\geq\frac{CR}{2\alpha}
    \end{split}
    \end{equation*}
for any $h$ big enough to ensure that $h\geq \frac{2\alpha }{CR}$.  A contradiction then follows.
\end{proof}
\begin{proposition}\label{mongequivdef}
    Let $K:\rr^n\times\rr^m\longrightarrow\rr$ be such that $K\in\mathcal K(H,\Om)$. A function $u \in C(\Om)$ is a \emph{Monge solution} (resp. \emph{subsolution}, \emph{supersolution}) to 
\eqref{eq:HJ}
in $\Omega$ if and only if 
\begin{equation}
\liminf_{x \to x_0} \frac{u(x) - u(x_0) + \olk(x_0, x)}{d_{ \Om}(x_0,x)}
= 0 \quad (\text{resp. } \geq, \leq)
\end{equation}
for any $x_0\in\Om$.
\end{proposition}
\begin{proof}
    It suffices to observe that, thanks to Lemma \ref{equivsigma} and the definition of $\mathcal K(H,\Om)$, for any $x_0\in\Om$ there exists $r>0$ such that
    \begin{equation*}
        \ol(x_0,x)=\olk(x_0,x)
    \end{equation*}
    for any $x\in B_r(x_0,d_\Om)$ and for any $K\in \mathcal K(H,\Om)$.
\end{proof}

Thanks to the results of Section \ref{olprop}, we are in position to prove Theorem \ref{mongedir}. The proof of this result is inspired by \cite{monge}.
\begin{proof}[Proof of Theorem \ref{mongedir}]
Let $K:\rr^n \times\mathbb{R}^m\longrightarrow\rr$ be as in the statement. First, notice that, since $g$ is bounded, then $w$ is well-defined. Fix $x,z\in\Om$ and, for any $h\in\mathbb N_+,$ let $y_h\in\partial\Om$ be such that $w(z)\geq\olk(z,y_h)+g(y_h)-\frac{1}{h}$. Then
\begin{equation*}
    w(x)-w(z)\leq \olk(x,y_h)-\olk(z,y_h)+\frac{1}{h}\leq \olk(x,z)+\frac{1}{h}\leq\ol(x,z)+\frac{1}{h}\leq\alpha d_\Om(x,z)+\frac{1}{h}.
\end{equation*}
Letting $h\to\infty$, and since $w(z)-w(x)$ can be estimated similarly, we conclude that $w\in \lip (\Om,d_\Om)$.
 Fix $x\in\partial\Om$. Then, by definition of $w$, it follows that $w(x)\leq g(x)$. On the other hand, if $y\in\partial\Om$, \eqref{compcond} implies that
    \begin{equation*}
        \olk(x,y)+g(y)\geq g(x),
    \end{equation*}
and so, taking the infimum over $\partial \Om$, we conclude that $w(x)\geq g(x)$. Therefore $w= g$ on $\partial\Om$.
Let now $x\in\partial\Om$ and let $(x_h)_h\subseteq\Om$ be such that $x_h\to x$ as $h\to\infty$. Then
\begin{equation*}
    w(x_h)-w(x)\leq \olk(x_h,x)+g(x)-g(x)\leq\alpha d_{ \G}(x_h,x)
\end{equation*}
and there exists $(y_h)_h \subseteq \partial \Om$ such that
\begin{equation*}
    w(x)-w(x_h)\leq g(x)-g(y_h)-\olk(x_h,y_h)+\frac{1}{h}\leq \olk(x,y_h)-\olk(x_h,y_h)+\frac{1}{h}\leq\olk(x,x_h)+\frac{1}{h}.
\end{equation*}
Hence we conclude that $w\in C(\overline\Om)$.
Let us show that $w$ is a Monge subsolution. To this aim, let $x_0\in\Om$ and let $(x_h)_h\subseteq\Om$ be such that $x_h\to x_0$ as $h\to\infty$. For any $h\in\mathbb N$, by definition of $w$, there exists $y_h\in\partial\Om$ such that
\begin{equation*}
    w(x_h)\geq \olk(x_h,y_h)+g(y_h)-\frac{\|x_0^{-1}\cdot x_h\|}{h}.
\end{equation*}
Therefore we infer that
\begin{equation*}
    \frac{w(x_h)-w(x_0)+\olk(x_0,x_h)}{\|x_0^{-1}\cdot x_h\|}\geq\frac{\olk(x_0,y_h)+g(y_h)-w(x_0)}{\|x_0^{-1}\cdot x_h\|}-\frac{1}{h}\geq-\frac{1}{h}.
\end{equation*}
Letting $h\to\infty$, being the sequence $(x_h)_h$ arbitrary and recalling Proposition \ref{mongequivdef}, we infer that $w$ is a Monge subsolution. Conversely, let $x_0\in\Om$ and assume without loss of generality that $B_{\frac{1}{h}}(x_0,\d)\subseteq\Om$ for any $h\in\mathbb N_+$ large enough. Fix such an $h$ and choose $y_h\in\partial\Om$ such that
\begin{equation*}
    w(x_0)\geq \olk (x_0,y_h)+g(y_h)-\frac{1}{h^2}.
\end{equation*}
Moreover, for any $h$, let $\gamma_h:[0,T_h]\longrightarrow \rr^n$ be a sub-unit curve with the property that $\gamma_h(0)=x_0$, $\gamma_h(T_h)=y_h$ and 
\begin{equation*}
    \olk (x_0,y_h)\geq\int_0^{T_h}\sigma_K^\star(\gamma(t),\Dot\gamma(t))\,dt-\frac{1}{h^2}.
\end{equation*}
Pick $t_h\in(0,T_h)$ such that $\gamma(t_h)\in\partial B_{\frac{1}{h}}(x_0,\d)$ and set $x_h:=\gamma(t_h)$. Then clearly $x_h\to x_0$ as $h\to\infty$ and therefore, by definition of $w$ and the choice of $(\gamma_h)_h$, we infer that
\begin{equation*}
    w(x_h)-w(x_0)+\olk (x_0,x_h)\leq \olk(x_h,y_h)-\olk(x_0,y_h)+\olk(x_0,x_h)+\frac{1}{h^2}\leq \frac{2}{h^2}.
\end{equation*}
Noticing that $\|x_0^{-1}\cdot x_h\|=\frac{1}{h}$, we conclude that
\begin{equation*}
    \liminf_{h\to\infty}\frac{w(x_h)-w(x_0)+\olk (x_0,x_h)}{\|x_0^{-1}\cdot x_h\|}\leq\liminf_{h\to\infty}\frac{2}{h}=0,
\end{equation*}
and so $w$ is a Monge supersolution.
\end{proof}
\section{Monge and Viscosity Solutions}\label{equivsect}
In this section we show that, as in the Euclidean setting (cf. \cite{NewSu,monge}), when $H$ is continuous the notions of Monge and viscosity solution coincide. We begin to prove that Monge solutions are viscosity solutions.

\begin{proposition}\label{firstimpl}
Let $H$ be continuous. If $u \in C(\Om)$ is a Monge subsolution (resp. supersolution) to \eqref{eq:HJ}, then $u$ is a viscosity subsolution (resp. supersolution) to \eqref{eq:HJ}.
\end{proposition}

\begin{proof}
    Let $u$ be a Monge supersolution to \eqref{eq:HJ}, fix $x_0\in\Om$ and $p\in\partial_X^-u(x_0)$. Then it follows that 
    \begin{equation*}
        0\geq\liminf_{x\to x_0}\frac{u(x)-u(x_0)+\ol(x_0,x)}{\|x_0^{-1}\cdot x\|}\geq \liminf_{x\to  x_0}\frac{\langle p,\pi(x_0^{-1}\cdot x)\rangle+\ol(x_0,x)}{\|x_0^{-1}\cdot x\|}
    \end{equation*}
Let $(x_h)_h$ be a minimizing sequence for the right hand side. Let us set $t_h:=\|x_0^{-1}\cdot x_h\|$ and $\xi_h:=\frac{1}{t_h}\pi(x_0^{-1}\cdot x_h)$. In this way, $t_h\to 0^+$ when $h\to\infty$. For any $h\in\mathbb N$, let $\eta_h\in\rr^{n-m}$ be such that
\begin{equation*}
    \delta_{\frac{1}{t_h}}(x_0^{-1}\cdot x_h)=(\xi_h,\eta_h),
\end{equation*}
 where we recall that by $\delta_{\frac{1}{t_h}}$ we mean intrinsic dilations as in Section \ref{carnotgroups}. 
By construction, $(\delta_{\frac{1}{t_h}}(x_0^{-1}\cdot x_h))_h$ is bounded. Then there exists $\xi\in\rr^m$ and $\eta\in\rr^{n-m}$ such that, up to a subsequence, $(\xi_h,\eta_h)\to(\xi,\eta)$ as $h\to\infty$.
Then, by Proposition \ref{stimapermongeimplicavisc} and the choice of $(x_h)_h$, we infer that
\begin{equation*}
    \begin{split}
        \liminf_{x\to x_0}\frac{\langle p,\pi (x_0^{-1}\cdot x)\rangle +\ol(x_0,x)}{\|x_0^{-1}\cdot x\|}&=\liminf_{h\to\infty}\left(\langle p,\xi_h\rangle +\frac{\ol(x_0,x_h)}{t_h}\right)\\
        &=\langle p,\xi\rangle +\liminf_{h\to\infty}\frac{\ol(x_0,x_0\cdot\delta_{t_h}(\xi_h,\eta_h))}{t_h}\\
        &=\langle p,\xi\rangle +\liminf_{h\to\infty}\frac{\ol(x_0,x_0\cdot\delta_{t_h}(\xi,\eta))}{t_h}\\
        &\geq \langle p,\xi\rangle +\sigma^\star(x_0,\xi).
    \end{split}
\end{equation*}
Therefore we conclude that
$\langle-\xi,p\rangle\geq\sigma^\star(x_0,\xi).$
If it was the case that $H(x_0,p)<0$, then 
$p$ is an interior point of $Z(x_0)$.
But then $\langle-\xi,p\rangle<\sigma^\star(x_0,\xi)$, since $q\mapsto\langle-\xi,q\rangle$ is a linear and non-constant, and so it achieves its maximum on $\partial Z(x_0)$. A contradiction then follows. \\
Assume now that $u$ is a Monge subsolution to \eqref{eq:HJ}, let $x_0\in\Om$ and $p\in\partial_X^+u(x_0)$. Assume by contradiction that $H(x_0,p)>0$. Hence, by Hahn--Banach Theorem, there exists $\xi\in \mathbb{S}^{m-1}$ such that 
    $\langle -\xi,p\rangle>\sigma^\star (x_0,\xi).
$
For any $h\in\mathbb N\setminus\{0\}$, let $x_h:=x_0\cdot\delta_{t_h}(\xi,0)$, where $(t_h)_h\subseteq(0,1)$ goes to $0$ as $h\to\infty$. Then $x_h\to x_0$ as $h\to\infty$, and moreover $x_0^{-1}\cdot x_h=(t_h\xi,0)$. Therefore, being $u$ a Monge subsolution, it follows that 
\begin{equation*}
\begin{split}
        0&\leq \liminf_{x\to x_0}\frac{u(x)-u(x_0)+\ol(x_0,x)}{\|x_0^{-1}\cdot x\|}\\
        &\leq \liminf_{x\to x_0}\frac{\langle p,\pi(x_0^{-1}\cdot x)\rangle+ \ol(x_0,x)}{\|x_0^{-1}\cdot x\|}\\
        &\leq\liminf_{h\to\infty}\frac{\langle p,\pi(x_0^{-1}\cdot x_h)\rangle+\ol(x_0,x_h)}{\|x_0^{-1}\cdot x_h\|}\\
        &=\langle p,\xi\rangle+\liminf_{h\to\infty}\frac{\ol(x_0,x_0\cdot\delta_{t_h}(\xi,0))}{t_h}.
\end{split}
\end{equation*}
Let us set $\gamma:[0,1]\longrightarrow\Om$ by $\gamma(t):=x_0\cdot\delta_t(\xi,0)$. Notice that $\Dot\gamma(t)\equiv\xi$, whence $\gamma$ is horizontal. Moreover, since $\xi\in \mathbb{S}^{m-1}$, $\gamma$ is sub-unit. Finally,  $\gamma(0)=x_0$ and $\gamma(t_h)=x_h$. Hence, since the continuity of $H$ implies the continuity of $\sigma^\star(\cdot,\xi)$,  we infer that
\begin{equation*}
\liminf_{h\to\infty}\frac{\ol(x_0,x_0\cdot\delta_{t_h}(\xi,0))}{t_h}\leq\liminf_{h\to\infty}\ave_0^{t_h}\sigma^\star(\gamma(t),\xi)\,dt=\sigma^\star(x_0,\xi).
\end{equation*}
Therefore we conclude that $\langle-\xi,p\rangle\leq\sigma^\star(x_0,\xi)$, a contradiction.

\end{proof}

In order to prove the converse implication, we need some preliminary results.

\begin{proposition}\label{subislipvisc}
    Let $H$ be continuous.
    Let $u\in C(\Om)$ and assume that $u$ is a viscosity subsolution to \eqref{eq:HJ}. Then $u\in W^{1,\infty}_{X,\loc}(\Om)$. 
\end{proposition}
\begin{proof}
    Let $x_0\in\Om$ and $p\in \partial_X^+ u(x_0)$ with $p\neq 0$. Then $H(x_0,p)\leq 0$, which implies that $p\in Z(x_0)$. Therefore it holds that $|p|\leq \alpha$ by $(H_3)$.
    Hence $u$ is a viscosity subsolution to  \begin{equation*}
        |Xu|\leq\alpha
    \end{equation*}
    on $\Om$.
   Thanks to Proposition \ref{euclivscarnot} and \cite[Proposition 2.1]{soravia}, we conclude that $u\in W^{1,\infty}_{X,\loc}(\Om)$.
\end{proof}

 \begin{proposition}\label{MonforLipv}
Assume that $H$ is continuous. If $u$ is a viscosity subsolution to \eqref{eq:HJ} in $\Om$, then
\begin{equation}\label{liprisps}
u(x)-u(y) \leq d_{\sigma^\star}(x,y) 
\end{equation}
for any $x,y\in \Om$. 
\end{proposition}
\begin{proof}
 Let $x,y\in\Om$. If $x=y$ the thesis is trivial. If instead $x\neq y$, let $\gamma :[0,T]\longrightarrow\Om$ be a sub-unit curve such that $\ga(0)=x$ and $\gamma (T)=y$ for some $T>0$.
Thanks to Proposition \ref{subislipvisc} and Proposition \ref{equivsol} we know that $u\in W^{1,\infty}_{X,\loc}(\Om)$ and that 
\begin{equation}\label{hzzzz}
    H(z,Xu(z))\leq 0
\end{equation}
for almost every $z\in\Om$.
Let $N$ be a Lebesgue negligible subset of $\Om$ containing all the non-Lebesgue points of $Xu$ and all the points where \eqref{hzzzz} does not hold. Then, in view of \cite[Lemma 2.7]{PVW}, we infer that
   $ H(z,p)\leq 0$
for any $z\in\Om$ and for any $p\in\partial_{X,N}u(z)$. Therefore, in particular,
\begin{equation}\label{pinz}
    p\in Z(z)
\end{equation}
for any $z\in\Om$ and for any $p\in\partial_{X,N}u(z)$.
Hence, thanks to the definition of $\sigma^\star$, \eqref{pinz} and Proposition \ref{w2}, we conclude that
\begin{equation*}
    u(x)-u(y) = u(\gamma(0))-u(\gamma(T))=-\int_0^T\frac{d (u \circ \gamma)(t)}{d t}\,dt=\int_0^T\langle \Dot\gamma(t),-\vartheta(t)\rangle\,dt\leq\int_0^T\sigma^\star (\gamma(t),\Dot\gamma(t))\,dt,
\end{equation*}
 where $\vartheta$ is as in Proposition \ref{w2}. Since $\gamma$ is arbitrary, the thesis follows.
\end{proof}
\begin{proposition}\label{secondimplsub}
    Let $H$ be a continuous Hamiltonian satisfying {\rm{(H)}}. Let $u\in C(\Om)$ be a viscosity subsolution to \eqref{eq:HJ}. Then $u$ is a Monge subsolution to \eqref{eq:HJ}. 
\end{proposition}
\begin{proof}
    Let $x_0\in\Om$. Hence, in view of \eqref{liprisps}, 
    we infer that
    \begin{equation*}
        u(x)-u(x_0)+\ol(x_0,x)\geq 0
    \end{equation*}
    for any $x\in\Om$, from which the thesis easily follows.
\end{proof}
In order to prove that viscosity supersolutions are Monge supersolutions, we argue as in \cite{NewSu}. To this aim, we combine Theorem \ref{mongedir} and Proposition \ref{firstimpl} to show the solvability of the Dirichlet problem associated to a continuous Hamiltonian in the setting of viscosity solutions.
\begin{theorem}\label{viscodir}
    Let $H$ be a continuous Hamiltonian satisfying {\rm{(H)}}, and let $g\in C(\partial\Om)$ be bounded and such that \eqref{compcond} holds.  Then the function $w$ defined by \eqref{hopflax} is a viscosity solution to the Dirichlet problem 
\begin{align*}
    H(x,Xw) &= 0 \quad \mbox{in } \Om \\
    w &= g \quad \mbox{on } \partial \Om.
\end{align*}
\end{theorem}
We need also the following sub-Riemannian comparison principle, whose proof is inspired by \cite{barles}.
\begin{proposition}\label{viscocomp}
Let $\Om$ be a bounded domain. Assume that $H$ is continuous and satisfies {\rm (H)}. Assume that $u\in C(\overline{\Om})\cap W^{1,\infty}_{X,\loc}(\Om)$ is a viscosity subsolution to \eqref{eq:HJ} on $\Om$ and that $v\in C(\overline\Om)$ is a viscosity supersolution to \eqref{eq:HJ} on $\Om$. If $u\leq v$ on $\partial\Om$, then $u\leq v$ on $\overline\Om$. 
\end{proposition}
\begin{proof}
    We can assume without loss of generality that $u,v>0$. Let us fix $\delta\in (0,1)$ and set $w:=\delta u$. Clearly $w\in C(\overline\Om)\cap W_{X,\loc}^{1,\infty}(\Om)$ and $w\leq v$ on $\partial\Om$. If we prove that $w\leq v$ on $\Om$, then the thesis follows letting $\delta\to 1$. \\
    \textbf{Step 1.}    
    We first claim that for any $\tilde\Om\Subset\Om$ there exists $\eta>0$ such that
   $w$ is a viscosity subsolution to 
   $$H(x,Xw)+\eta=0 \quad \text{on} \quad \tilde\Om.$$  If it was not the case, then there exists $\tilde\Om\Subset\Om$ and sequences $(x_h)_h\subseteq\tilde\Om$, $(p_h)_h\subseteq \rr^m$ such that $p_h\in\partial^+_X w(x_h)$ and 
   \begin{equation*}
       H(x_h,p_h)+\frac{1}{h} > 0
   \end{equation*}
   for any $h\in\mathbb N_+$. Since by assumption $Z(x_h)\subseteq\hat B_\alpha(0)$ for any $h\in\mathbb N_+$, then we can assume up to a subsequence that $x_h\to \Tilde{x}\in\Om$ and $p_h\to \Tilde{p}\in\rr^m$. Being $H$ continuous, we infer that 
       $H(\Tilde{x}, \Tilde{p})\geq 0.$
   On the other hand, notice that
   $\frac{p_h}{\delta}\in\partial^+_Xu(x_h)$ for any $h\in\mathbb N_+$, and so, being $u$ a subsolution, we infer that
$H\left(x_h,\frac{p_h}{\delta}\right)\leq 0.$
   Since $H$ is continuous, we conclude that
   \begin{equation*}
         H\left(\Tilde{x},\frac{\Tilde{p}}{\delta}\right)\leq 0.
   \end{equation*}
The last equation implies that $\frac{\Tilde p}{\delta}\in Z(\Tilde x)$. But then, being $Z(\Tilde x)$ convex and since $|\Tilde p|<\frac{|\Tilde p|}{\delta}$, we conclude that 
$\Tilde p$ is an interior point of $Z(\bar x)$, and so $H(\Tilde x,\Tilde p)<0$, a contradiction.\\
\textbf{Step 2.} Let us define $M:=\max_{\overline\Om} (w-v) $, and assume by contradiction that $M>0.$ Let us define, for any $\varepsilon\in(0,1)$,
\begin{equation*}
    \varphi_\varepsilon(x,y):=w(x)-v(y)-\frac{\d(x,y)^{2k!}}{\varepsilon ^2},
\end{equation*}
where $\d$ is the well-known homogeneous distance introduced in \eqref{od}.
Being $\varphi_\varepsilon$ continuous on $\overline\Om\times\overline\Om$, there exists $(x_\varepsilon,y_\varepsilon)\in\overline\Om\times\overline\Om$ such that 
\begin{equation*}
M_\varepsilon:=\max_{\overline\Om\times\overline\Om}\varphi_\varepsilon=\varphi_\varepsilon(x_\varepsilon,y_\varepsilon).
\end{equation*}
\textbf{Step 3}. We claim the following facts.
\begin{itemize}
    \item [$(i)$] $M_\varepsilon\to M$ as $\varepsilon\to 0$.
    \item[$(ii)$] $w(x_\varepsilon)-v(y_\varepsilon)\to M$ as $\varepsilon \to 0$.
    \item[$(iii)$] $\frac{\d(x_\varepsilon,y_\varepsilon)^{2k!}}{\varepsilon^2}\to 0$ as $\varepsilon\to 0$.
    \item [$(iv)$]Let us set
    \begin{equation*}
        p_\varepsilon:=\frac{(2k!)\d(x_\varepsilon,y_\varepsilon)^{2k!-1}X\d(x_\varepsilon,y_\varepsilon)}{\varepsilon^2}.
    \end{equation*}
    Then $(p_\varepsilon)_\varepsilon$ is bounded.
    \item [$(v)$] There exists $\tilde\Om\Subset\Om$ such that $x_\varepsilon,y_\varepsilon\in\tilde\Om$ for any $\varepsilon$ small enough.
\end{itemize}
 Indeed, since from the choice of $(x_\varepsilon,y_\varepsilon)$ it is easy to see that $M\leq M_\varepsilon$ for any $\varepsilon\in (0,1)$. Let us set $R:=\max\{\|w\|_{\infty},\|v\|_\infty\}$.
Then we have that
\begin{equation*}
    M\leq 2R-\frac{\d(x_\varepsilon,y_\varepsilon)^{2k!}}{\varepsilon^2}.
\end{equation*}
Since we assumed that $M>0$, we infer that
\begin{equation*}
    \frac{\d(x_\varepsilon,y_\varepsilon)^{2r!}}{\varepsilon^2}\leq 2R.
\end{equation*}
This implies in particular that $\d(x_\varepsilon,y_\varepsilon)\to 0$ as $\varepsilon\to 0$. This fact, together with the compactness of $\overline\Om$, allows to assume up to a subsequence that there exists $\bar x\in\overline\Om$ such that
\begin{equation}\label{samelimit}
    \lim_{\varepsilon\to 0}\d(x_\varepsilon ,\bar x)= \lim_{\varepsilon\to 0}\d(y_\varepsilon ,\bar x)=0.
\end{equation}
Moreover, notice that $M\leq M_\varepsilon$ implies that
    $M\leq w(x_\varepsilon)-v(y_\varepsilon)$
for any $\varepsilon>0$. This last inequality, together with \eqref{samelimit}, implies that
\begin{equation}\label{stima}
    M\leq\liminf_{\varepsilon\to 0}w(x_\varepsilon)-v(y_\varepsilon)\leq\limsup_{\varepsilon\to 0}w(x_\varepsilon)-v(y_\varepsilon)\leq M.
\end{equation}
This proves $(ii)$. The fact that $M\leq M_\varepsilon$, combined with \eqref{stima}, allows to conclude that
\begin{equation*}
    M\leq\liminf_{\varepsilon\to 0}M_\varepsilon\leq \lim_{\varepsilon\to 0}w(x_\varepsilon)-v(y_\varepsilon)=M.
\end{equation*}
This proves $(i)$ and $(iii)$. To prove $(v)$, it suffices to observe that
\begin{equation*}
    M=\lim_{\varepsilon\to 0}w(x_\varepsilon)-v(y_\varepsilon)=w(\bar x)-v(\bar x),
\end{equation*}
and thus, recalling that $M>0$ and that $w\leq v$ on $\partial\Om$, $(v)$ follows. Finally, we prove $(iv)$. 
Indeed, notice that, in view of the choice of $x_\varepsilon,y_\varepsilon$, then
\begin{equation*}
    w(y_\varepsilon)-v(y_\varepsilon)=\varphi_\varepsilon(y_\varepsilon,y_\varepsilon)\leq\varphi(x_\varepsilon,y_\varepsilon)=w(x_\varepsilon)-v(y_\varepsilon)-\frac{\d(x_\varepsilon,y_\varepsilon)^{2k!}}{\varepsilon^2},
\end{equation*}
which implies that
\begin{equation*}
    \frac{\d(x_\varepsilon,y_\varepsilon)^{2k!}}{\varepsilon^2}\leq w(x_\varepsilon)-w(y_\varepsilon)\leq C\d(x_\varepsilon,y_\varepsilon),
\end{equation*}
where $C>0$ is the $\d-$Lipschitz constant of $w$ on $\tilde\Om$. Therefore 
\begin{equation*}
    \frac{\d(x_\varepsilon,y_\varepsilon)^{2k!-1}}{\varepsilon^2}\leq C.
\end{equation*} The proof is concluded noticing that $z\mapsto X\d(z_0,z)$ is bounded on $\Om\setminus\{z_0\}$ uniformly with respect to $z_0\in\Om$. \\
\textbf{Step 4.} Let us define 
\begin{equation*}
 \varphi^1_\varepsilon(y):=w(x_\varepsilon)-\frac{\d(x_\varepsilon,y)^{2k!}}{\varepsilon^2}
 \qquad\text{and}\qquad 
\varphi^2_\varepsilon(x):=v(y_\varepsilon)+\frac{\d(x,y_\varepsilon)^{2k!}}{\varepsilon^2}
\end{equation*}
for any $x,y\in\Om$. These are smooth functions on $\Om$. Moreover, $x_\varepsilon$ is a maximum point for $x\mapsto w(x)-\varphi_\varepsilon^2(x)$ and $y_\varepsilon$ is a maximum point for $y\mapsto -v(y)+\varphi_\varepsilon^1(y)$. Therefore, if $\eta>0$ is the constant coming from Step 1 and relative to $\tilde\Om$ as in $(v)$, then
\begin{equation*}
    H(x_\varepsilon,p_\varepsilon)+\eta\leq H(y_\varepsilon,p_\varepsilon).
\end{equation*}
Being $(p_\varepsilon)_\varepsilon$ bounded, we can assume that $p_\varepsilon\to\bar p$ as $\varepsilon\to 0$. Therefore we conclude from the previous inequality that
\begin{equation*}
    H(\bar x,\bar p)+\eta\leq H(\bar x,\bar p),
\end{equation*}
a contradiction.
\end{proof}
\begin{remark}
 Alternatively, the proof of Proposition \ref{viscocomp} could be carried out, following \cite[Lemma 2.7]{Barles94} and \cite[Lemma 6.3]{FaSi05}, by regularizing the viscosity subsolution $w$ via mollification. Instead, the above doubling variable argument avoids the otherwise necessary recourse to the  \emph{group convolution} as in \cite[Proposition 1.20]{FS}.
\end{remark}

\begin{proposition}\label{viscohenceMonge}
Let $H$ be continuous. Let $u\in C(\Om)$ be a viscosity supersolution to \eqref{eq:HJ}. Then $u$ is a Monge supersolution to \eqref{eq:HJ}. 
\end{proposition}
\begin{proof}
Let $u$ be as in the statement. If by contradiction  $u$ is not a Monge supersolution to \eqref{eq:HJ}, there exists $x_0 \in \Om$, $r>0$ and $\delta>0$ such that
\begin{equation}\label{rdelta}
    u(x) - u(x_0) + d_{\sigma^\star}(x_0, x) \geq \delta \|x_0^{-1}\cdot x\|
\end{equation}
for any $x\in B_r(x_0,\d)$. Notice that, without loss of generality, we can assume that $u(x_0)=0$. Set $\psi(x) = - d_{\sigma^\star}(x_0, x) + \delta r$. Notice that, as $B_r(x_0,\d)\Subset\Om$, then $H\in \mathcal K(H,B_r(x_0,\d))$, being $H$ extended as in \eqref{numeraperdopo}. Moreover, notice that
\begin{equation*}
    \psi(x)-\psi(y)=\ol(x_0,y)-\ol (x_0,x)\leq\ol (x,y)
\end{equation*}
for any $x,y\in\partial B_r(x_0,\d)$, and so \eqref{compcond} is satisfied by $\psi$.
Therefore we know from Theorem \ref{viscodir} that, if we define $w: \overline{B_r(x_0,\d)}\longrightarrow\rr$ as in \eqref{hopflax} with $\Om=B_r(x_0,\d)$ and $g=\psi$, then $w\in C(\overline{B_r(x_0,\d)})$ and $w$ solves in the viscosity sense the Dirichlet problem
\begin{align*}
    H(x,Xw) &= 0 \quad \mbox{in } B_r(x_0,\d) \\
    w &= \psi \quad \mbox{on } \partial B_r(x_0,\d).
\end{align*}
Moreover, in view of \eqref{rdelta}, $u\geq \psi$ on $\partial B_r(x_0,\d)$. Therefore, recalling that $w\in C(\overline {B_r(x_0,\d)})\cap W^{1,\infty}_{X,\loc}(B_r(x_0,\d))$, we conclude from Proposition \ref{viscocomp} that $w(x_0)\leq u(x_0)=0$, but this is impossible, since $
    w(x_0) = \delta r > 0.
$
\end{proof}
\begin{proof}[Proof of Theorem \ref{introequiv}]
    It follows from Proposition \ref{firstimpl}, Proposition \ref{secondimplsub} and Proposition \ref{viscohenceMonge}.
\end{proof}
\section{Comparison Principle and Stability}\label{lastsect}
\subsection{Comparison Principle}
In this section we prove Theorem \ref{comparison}. This result, as customary, yields uniqueness for the Dirichlet problem associated to \eqref{eq:HJ}.
The proof of Theorem \ref{comparison}, strongly inspired by \cite{Davinitesi}, is based on the validity of the following two properties of Monge subsolutions.
\begin{proposition}\label{subislip}
    Let $u\in C(\Om)$. Assume that $u$ is a Monge subsolution to \eqref{eq:HJ}. Then $u\in W^{1,\infty}_{ X,\loc}(\Om)$.
\end{proposition}
\begin{proof}
    Assume that $u\in C(\Om)$ is a Monge subsolution to \eqref{eq:HJ}. Then 
    \begin{equation*}
        \liminf_{x\to x_0} \frac{u(x)-u(x_0)+\alpha d_{\Om}(x_0,x)}{\|x_0^{-1}\cdot x\|}\geq\liminf_{x\to x_0}\frac{u(x)-u(x_0)+\ol (x_0,x)}{\|x_0^{-1}\cdot x\|}\geq 0
    \end{equation*}
    for any $x_0\in\Om$. Let $K(x,\xi):=|\xi|-\alpha$. Then $\sigma^\star_K(x,\xi)=\alpha |\xi|$ and $d_{\sigma^\star_K}(x,y)=\alpha d_{\Om}(x,y)$. This implies that $u$ is a Monge subsolution to 
    \begin{equation}\label{lipeq2}
        |Xu|\leq\alpha
    \end{equation}
    on $\Om$. Since $K$ is continuous, then $u$ is also a viscosity subsolution to \eqref{lipeq2}, in view of Proposition \ref{firstimpl}. The conclusion follows as in the proof of Proposition \ref{subislipvisc}.
\end{proof}
 \begin{proposition}\label{MonforLip}
If $u$ is a Monge subsolution to \eqref{eq:HJ} in $\Om$, then for any $x_0\in\Om$ there exists $r>0$ such that
\begin{equation*}
u(x)-u(y) \leq d_{\sigma^\star}(x,y) 
\end{equation*}
for any $x, y\in \overline{B_r(x_0,d_\Om)}$. 
\end{proposition}
\begin{proof}
 Let $r>0$ be as in Proposition \ref{locgeo}. Then in particular $B_r(x_0,d_\Om)\Subset\Om$. Moreover, since $u$ and $\ol$ are continuous on $\overline{B_r(x_0,d_\Om)}$, it suffices to consider points in $B_r(x_0,d_\Om)$. Let $x,y\in B_r(x_0,d_\Om)$. If $x=y$ the thesis is trivial. If instead $x\neq y$, in view of Proposition \ref{locgeo} there exists a sub-unit curve $\gamma :[0,T]\longrightarrow\Om$ such that $\ga(0)=x$, $\gamma (T)=y$ for some $T>0$ and $\gamma$ is optimal for $\ol (x,y)$ in the sense of Section \ref{lentandgeodistancesection}. In particular, \eqref{onlitool} holds for $\gamma$.
Set $f(t) \coloneqq d_{\sigma^\star}(x_0, \gamma(t))$ and $g(t):= u(\gamma(t))$. Therefore Proposition \ref{subislip} implies that both $f,g\in W^{1,\infty}_{\loc}(0,T)$.
We infer that the derivative of $f+g$ exists almost everywhere on $(0,T)$. To conclude, it suffices to show that it is non-negative. To this aim, recalling that $u$ is a Monge subsolution to \eqref{eq:HJ} and by \eqref{onlitool}, we observe that 
\begin{equation*}
  \begin{split}
      \frac{d}{dt}(f+g)\Big|_{t=t_0}&=\lim_{h\to 0^+}\frac{g(t_0+h)-g(t_0)+f(t_0+h)-f(t_0)}{h}\\
      &=\lim_{h\to 0^+}\frac{u(\ga(t_0+h))-u(\ga(t_0))+\ol(x_0,\ga(t_0+h))-\ol(x_0,\ga(t_0))}{\|\ga(t_0)^{-1}\cdot\ga(t_0+h)\|}\cdot\frac{\|\ga(t_0)^{-1}\cdot\ga(t_0+h)\|}{h}\\
      & = \liminf_{h\to 0^+}\frac{u(\ga(t_0+h))-u(\ga(t_0))+\ol(\ga(t_0),\ga(t_0+h))}{\|\ga(t_0)^{-1}\cdot\ga(t_0+h)\|}\cdot\frac{\|\ga(t_0)^{-1}\cdot\ga(t_0+h)\|}{h}\\
      &\geq 0
  \end{split}
\end{equation*}
for almost every $t_0\in (0,T)$.  Finally, integrating $\frac{d}{dt}(f+g)$ in $[0,T]$ we get the result.
\end{proof}
\begin{lemma}\label{hk}
  Let $\Om\subseteq\G$ be a bounded domain. Let $H, K:\Om\times\rr^m\longrightarrow\rr$ satisfy {\rm (H)}, and assume that there exists $\delta\in(0,1)$ such 
    \begin{equation}\label{inclu}
        Z_K(x)\subseteq \delta Z_H(x)
    \end{equation}
    for any $x\in\Om$. Assume that $u\in C(\overline\Om)$ is a Monge subsolution to $K(x,Xu)=0$ and that $v\in C(\overline\Om)$ is a Monge supersolution to $H(x,Xv)=0$. If $u\leq v$ on $\partial\Om$, then $u\leq v$ on $\overline\Om$.
\end{lemma}
\begin{proof}
Assume by contradiction that there exists $x_0\in\Om$ such that $u(x_0)>v(x_0)$.
    Let us define $\tilde H,\tilde K$ by
    \begin{equation*}
        \tilde H(x,\xi):=
        \begin{cases}
            H(x,\xi)\qquad\text{if }(x,\xi)\in\Om\times\rr^m\\
            |\xi|-\alpha\qquad\text{otherwise}
        \end{cases}
    \end{equation*}
    and
      \begin{equation*}
        \tilde K(x,\xi):=
        \begin{cases}
            K(x,\xi)\qquad\text{if }(x,\xi)\in\Om\times\rr^m\\
            |\xi|-\frac{1}{\alpha}\qquad\text{otherwise}.
        \end{cases}
    \end{equation*}
 Then $\tilde H\in\mathcal{K}(H,\Om)$ and $\tilde K\in\mathcal K(K,\Om)$. Notice that, since $\tilde H,\tilde K$ are defined on the whole $\mathbb G\times\rr^m$, then $\oth,\otk$ are geodesic distances in the sense of Section \ref{lentandgeodistancesection} (cf. \cite{esspinpas}). Moreover, \eqref{inclu} and the definition of $\tilde H,\tilde K$ imply that
 \begin{equation}\label{includue}
        Z_{\tilde K}(x)\subseteq \delta Z_{\tilde H}(x)
    \end{equation}
    holds for any $x\in\Om$. We claim that there exists $\varepsilon>0$ such that 
    \begin{equation*}
        f_\varepsilon(x,y):=u(x)-v(y)-\frac{\oth(x,y)^2}{\varepsilon}
    \end{equation*}
    achieves its maximum over $\overline\Om\times\overline\Om$ on $\Om\times\Om$. If not, then for any $h\in\mathbb N_+$ there exists $(x_h,y_h)\in(\overline\Om\times\overline\Om)\setminus\Om\times\Om$ which realizes the maximum for $f_{\frac{1}{h}}$. Up to a subsequence, we can assume that $x_h\to\bar x$ and that $y_h\to\bar y$. Then either $\bar x\in\partial\Om$ or $\bar y\in\partial\Om$. Let us assume for definiteness that the first
case occurs. Notice that 
    \begin{equation}\label{eqau}
        0<f_{\frac{1}{h}}(x_0,x_0)\leq  u(x_h)- v(y_h)-h\oth(x_h,y_h)^2.
    \end{equation}
    Therefore $h\oth(x_h,y_h)^2$ is bounded, and hence $d_\G (x_h,y_h)\to 0$. This implies that $\bar x=\bar y$. Hence, noticing that $f_\frac{1}{h}(x_0,x_0)$ does not depend on $h$, \eqref{eqau} implies that $u(\bar x)>v(\bar x)$, which is impossible since $\bar x\in\partial\Om$. Let then $(\tilde x,\tilde y)\in\Om\times\Om$ be a maximum point for $f_\varepsilon $. Since $\oth$ is geodesic, there exists a sub-unit curve $\ga:[0,T]\longrightarrow\G$  such that 
    $\ga(0)=\tilde x,$ $\ga(T)=\tilde y$ and, recalling \eqref{onlitool}, 
\begin{equation}\label{onlitoollemma72}
    \oth(\tilde x,\tilde y)=\oth(\tilde x,\gamma(t))+\oth(\gamma(t),\tilde y)
\end{equation}
for any $t\in [0,T]$. Set 
    \begin{equation*}
        h(t):=\frac{1}{\varepsilon}(\oth(\tilde x,\tilde y)+\oth(\ga(t),\tilde y)).
    \end{equation*}
    We claim that $h(0)\leq \delta$. If $\tilde x=\tilde y$, the thesis is trivial. So assume $\tilde x\neq\tilde y$. Notice that $f_\varepsilon(\tilde x,\tilde y)\geq f_\varepsilon(\gamma(t),\tilde y)$ for any $t$ small enough to ensure that $\gamma(t)\in\Om$, and so 
    \begin{equation*}
        u(\tilde x)-u(\gamma(t))\geq  \frac{1}{\varepsilon}\left(\oth(\tilde x,\tilde y)^2-\oth(\gamma(t),\tilde y)^2\right)=h(t)(\oth(\tilde x,\tilde y)-\oth(\gamma(t),\tilde y))= h(t)\oth(\tilde x,\ga(t))
    \end{equation*}
    for any $t$ small enough, the last equality following in view of \eqref{onlitoollemma72}.
    Since $u$ is a Monge subsolution to $K(x,Xu)=0$, we can apply Proposition \ref{MonforLip} to infer that
    \begin{equation*}
        \otk(\tilde x,\ga(t))\geq h(t)\oth(\tilde x,\ga(t))
    \end{equation*}
    for any $t>0$ small enough. Moreover \eqref{includue} implies that $\otk(\tilde  x,\ga(t))\leq\delta\oth(\tilde x,\ga(t))$ for any $t\in [0,T]$. We conclude that 
    \begin{equation*}
        \delta\oth(\tilde x,\ga(t))\geq h(t)\oth(\tilde x,\ga(t))
    \end{equation*}
    for any $t>0$ small enough, which yields the claim. Noticing that $f_\varepsilon(\tilde x,\tilde y)\geq f_\varepsilon (\tilde x,y)$ for any $y\in\Om$, then
    \begin{equation*}
    \begin{split}
            v(\tilde y)-v(y)&=f_\varepsilon(\tilde x,y)-f_\varepsilon(\tilde x,\tilde y)+\frac{1}{\varepsilon}(\oth(\tilde x,y)^2-\oth(\tilde x,\tilde y)^2)\\
            & \leq \frac{1}{\varepsilon}(\oth(\tilde x,y)^2-\oth(\tilde x,\tilde y)^2)\\
            &\leq \frac{1}{\varepsilon}\left(\oth(\tilde x,y)+\oth(\tilde x,\tilde y)\right)\oth(\tilde y,y)\\
            &= \left(h(0)+\frac{1}{\varepsilon}\left(\oth(\tilde x,y)-\oth (\tilde x,\tilde y)\right)\right)\oth(\tilde y,y) \\
            &\leq  \left(\delta+\frac{\alpha}{\eps} d_\G( y,\tilde y)\right)\oth(\tilde y,y)\\
            &<\frac{1+\delta}{2}\oth(\tilde y,y)
    \end{split}
    \end{equation*}
    for any $y$ in a neighborhood of $\tilde y$, where the semi-last inequality follows from $h(0)\leq\delta$  and from Lemma \ref{prop:optimalpath}, while the last inequality follows provided that $y$ is sufficiently close to $\tilde y$ to ensure that
    \begin{equation*}
        d_\G(y,\tilde y)<\frac{\eps(1-\delta)}{2\alpha}.
    \end{equation*}
    Therefore we can conclude that 
    \begin{equation*}
        v(y)-v(\tilde y)+\oth(\tilde y,y)\geq \frac{1-\delta}{2}\oth(\tilde y,y)\geq\frac{1-\delta}{2\alpha}d_\G(\tilde y,y),
    \end{equation*}
    which is a contradiction since $v$ is a Monge supersolution to $H(x,Xv)=0$.
\end{proof}
\begin{proof}[Proof of Theorem \ref{comparison}]
    We follow  \cite[Theorem 5.8]{Davinitesi}. Up to replacing  $u$ with $u+c$ and $v$ with $v+c$ for a positive constant $c$ large enough, we may assume $u$ and $v$ positive in $\bar{\Omega}$. Let $\delta \in (0,1)$ and $H_{\delta} (x,\xi)=H(x,\frac{\xi}{\delta})$. Notice that $Z_{H_{\delta}}(x)= \delta Z_H (x)$ for each $x$ in $\Omega$, whence $d_{\sigma_{ H_\delta}^\star}=\delta d_{\sigma_H^\star}$. This implies that $\delta u$ is a Monge subsolution to $H_{\delta}(x, Xw)=0$. Moreover, $\delta u \le u \le v$ in $\partial \Omega$, so we can apply  Lemma \ref{hk} with $K=H_{\delta}$ to obtain that $\delta u \le v$ in $\Omega$. By letting $\delta \to 1^-$ we get the desired result.
\end{proof}
\subsection{Stability} Finally, following \cite{monge}, we prove Theorem \ref{stab}, which is the analogue of \cite[Theorem 6.4]{monge}.

\begin{proof}[Proof of Theorem \ref{stab}]
 Fix $x_0\in\Om$ and let $r>0$ be such that $B_r(x_0,d_\Om)\Subset\Om$ and Proposition \ref{locgeo} holds. Then $H_n\in \mathcal K(H_n,B_r(x_0,d_\Om))$ for any $n\in\mathbb N$ and $H_\infty\in \mathcal K(H_\infty,B_r(x_0,d_\Om))$, \ being $H_n$ and $H_\infty$ extended as in \eqref{numeraperdopo}. Moreover, in view of Proposition \ref{MonforLip}, $
     u_n(x)-u_n(y)\leq d_{\sigma^\star_{H_n}}(x,y)
$
 for any $n\in\mathbb N$ and for any $x,y\in\partial B_r(x_0,d_\Om)$. Hence, in view of Theorem \ref{mongedir},
 \begin{equation*}
     u_n(x)=\inf_{y\in\partial B_r(x_0,d_\Om)}\{d_{\sigma^\star_{H_n}}(x,y)+u_n(y)\}
 \end{equation*}
 for any $x\in\overline{B_r(x_0,d_\Om)}$ and any $n\in\mathbb N$. By the local uniform convergence assumptions we infer that 
  \begin{equation*}
     u_\infty(x)=\inf_{y\in\partial B_r(x_0,d_\Om)}\{d_{\sigma^\star_{H_\infty}}(x,y)+u_\infty(y)\}
 \end{equation*}
 for any $x\in\overline{B_r(x_0,d_\Om)}$, and so we conclude thanks to Theorem \ref{mongedir}.
\end{proof}
\begin{remark}
    The convergence condition in the hypotheses of Theorem \ref{stab} is based on the optical length functions rather than on the Hamiltonians. Arguing as in \cite{monge}, one can easily find sufficient conditions on the Hamiltonians in order to guarantee the local uniform convergence of the optical length functions.
\end{remark}

\bibliographystyle{abbrv} 
\bibliography{sub-finsler}

\end{document}